\documentclass[a4paper]{article}
\usepackage{amsmath,amssymb,amsthm,amsfonts}
\usepackage{enumerate,color, graphicx}
\usepackage{supertabular}
\usepackage{booktabs}
\usepackage{url}
\usepackage{xcolor}
\usepackage{colortbl}
\usepackage{subfig}
\usepackage[left=3.5cm, right=3.5cm, top=2.5cm, bottom=2.5cm]{geometry}

\newcommand{\R}{\ensuremath{\mathbb{R}}}

\newcommand{\N}{\ensuremath{\mathbb{N}}}

\numberwithin{equation}{section}

\newtheorem{thm}{Theorem}[section]
\newtheorem{prop}[thm]{Proposition}
\newtheorem{cor}[thm]{Corollary}
\newtheorem{lem}[thm]{Lemma}

\newtheorem{rem}[thm]{Remark}

\newtheorem{conjecture}[thm]{Conjecture}

\title{Courant-sharp Robin eigenvalues for the square:\\ the case of negative Robin parameter}

\author{K. Gittins
\footnote{ Universit\'e de Neuch\^atel, Institut de Math\'ematiques, Rue Emile-Argand 11, CH-2000 Neuch\^atel.
Email: \texttt{katie.gittins@unine.ch}}\,
and B. Helffer
\footnote{Laboratoire de Math\'ematiques Jean Leray, Universit\'e de Nantes, 2 rue de la Houssini\`ere, 44 322 Nantes CEDEX 3 - FRANCE.
Email: \texttt{Bernard.Helffer@univ-nantes.fr}}
}

\date{\today}

\begin{document}
	\maketitle
\begin{abstract}	
 We consider the cases where there is equality in Courant's nodal domain theorem for the Laplacian with a Robin boundary condition on the square.
We treated the cases where the Robin parameter $h>0$ is large, small in \cite{GHRS}, \cite{GHRS2} respectively.
In this paper we investigate the case where $h<0\,$.
\end{abstract}

\paragraph{MSC classification (2010):}    35P99, 58J50, 58J37.


\section{Introduction.}

We consider a bounded, connected, open set $\Omega \subset \R^m$, $m \geq 2$, with Lipschitz boundary
and $h \in \R$. The Robin eigenvalues of the Laplacian on $\Omega$ with parameter
$h$ are $\lambda_{k,h}(\Omega) \in \R$, $k \in \N$, $k \geq 1$, such that there exists a function $u_k \in
H^1(\Omega)$ which satisfies
\begin{align*}
-\Delta u_k(x) &= \lambda_{k,h}(\Omega)u_k(x)\,, \quad x \in \Omega\,, \notag \\
\frac{\partial}{\partial \nu} u_k(x) &+ h\, u_k(x) = 0\,,  \quad x \in \partial\Omega\,, 
\end{align*}
where $\nu$ is the outward-pointing unit normal to $\partial \Omega$.
It is well known that under these geometric constraints, the Robin Laplacian on $\Omega$ has discrete spectrum
\begin{equation*}
  \lambda_{1,h}(\Omega) \leq \lambda_{2,h}(\Omega) \leq \dots \,
\end{equation*}
and that one can find an orthonormal basis $(u_k)_{k\in \N}$ in $L^2(\Omega)$  such that $u_k$ is an eigenfunction associated with $\lambda_{k,h}$.
By the minimax principle, the Robin eigenvalue problem has a corresponding quadratic form:
\begin{equation*}
H^1(\Omega) \ni u \mapsto \int_\Omega |\nabla u|^2 + h \int_{\partial \Omega} |u_{\partial \Omega}|^2 d\sigma\, ,
\end{equation*}
where $u_{\partial \Omega}$ is the trace of $u$.
Hence the Robin eigenvalues are monotonically increasing with respect to $h$ for $h\in (-\infty,\infty)$.
In addition, each Robin eigenvalue with $h <0$ is smaller than the corresponding Neumann eigenvalue, denoted $\mu_k(\Omega) = \lambda_{k,0}(\Omega)$.

The Robin eigenvalues satisfy Courant's nodal domain theorem \cite{CH} which states that any eigenfunction corresponding to $\lambda_{k,h}(\Omega)$ has at most $k$ nodal domains. We are interested in the Courant-sharp
Robin eigenvalues of $\Omega$, that is the Robin eigenvalues $\lambda_{k,h}(\Omega)$ that have a corresponding eigenfunction with exactly $k$ nodal domains. As for the Dirichlet and Neumann eigenvalues, $\lambda_{1,h}(\Omega)$ and $ \lambda_{2,h}(\Omega)$ are Courant-sharp for all $h \in \R$.

We treat the particular example where $\Omega$ is the square $S = (-\frac{\pi}{2},\frac{\pi}{2})^2 \subset \R^2$. Our key question is whether it is possible to determine the Courant-sharp eigenvalues of the Robin Laplacian on $S$ with parameter $h<0$.

In previous work \cite{GHRS,GHRS2}, we considered the case where $h>0$. In \cite{GHRS}, we showed that there are finitely many Courant-sharp Robin eigenvalues when $h>0$.
\begin{thm}\label{thm:hpos1}
 Let $h\geq 0$.
 If $\lambda_{k,h}(S)$ is an eigenvalue of the Robin Laplacian on $S$ with parameter $h$ and $k \geq 520$, then it is not Courant-sharp.
\end{thm}
In addition, we proved that for $h$ sufficiently large, the only Courant-sharp Robin eigenvalues are for $k=1,2,4$.
\begin{thm}\label{thm:hpos2}
 There exists $h_1>0$ such that for $h\geq h_1$, the Courant-sharp cases for the Robin problem on $S$ are the same as those for $h=+\infty\,$  (i.e. the Dirichlet case).
\end{thm}

It was shown in \cite{HPS1} that the only Courant-sharp Neumann eigenvalues of the square are for $k=1,2,4,5,9$.
On the other hand, in \cite{GHRS2}, we proved the following theorem.
\begin{thm}\label{thm:hpos3}
 There exists $h_0>0$ such that for $0<h \leq h_0$, the Courant-sharp cases for the Robin problem on $S$ are the same, except the fifth one, as those for $h=0$ (i.e. the Neumann case)\,.
\end{thm}

The goal of the present paper is to investigate the case where $h<0$.
As $\lambda_{2,h}(S)=\lambda_{3,h}(S)$, it follows immediately that $\lambda_{3,h}(S)$ is not Courant-sharp for any $h < 0\,$. On the other hand, we prove the following result for the  fourth and fifth Robin eigenvalues of $S$ when $h<0$.

\begin{thm}\label{thm:5}
For $h<0$, the  fourth and  fifth eigenvalues of the Robin Laplacian on $S$ with parameter $h$, $ \lambda_{4,h}(S), \lambda_{5,h}(S)$, are Courant-sharp.
\end{thm}

In \cite{GHRS}, for the ninth Robin eigenvalue of $S$ when $h>0$, we proved that there exists $h_9^* >0$ such that $\lambda_9^*(S)$ is Courant-sharp for $0 \leq h \leq h_9^*$, and is not Courant-sharp for $h > h_9^*$.
For the ninth Robin eigenvalue of $S$ when $h<0$, we have the following proposition.

\begin{thm}\label{p:ninth}
There exists $h_9^* < 0$ such that the ninth eigenvalue of the Robin Laplacian on $S$ with parameter $h$, $\lambda_{9,h}(S)$, is Courant-sharp for $h_9^* \leq h \leq 0$ and is not Courant-sharp for $h < h_9^*$. Numerically, we have $h_9^* \approx -1.6293$.
\end{thm}

In Sections 7 and 8 of \cite{GHRS2}, we showed that if we start from the nodal set of a Neumann eigenfunction and perform a sufficiently small perturbation of $h \in \R$, then the number of nodal domains does not increase.
It is possible to show that for $h<0$, $\vert h \vert$ sufficiently small, the labelling of the Robin eigenvalues $\lambda_{k,h}(S)$ is the same as that for the Neumann eigenvalues $\lambda_{k,0}(S)$ (see Remark \ref{rem:label}).
Thus, for any Neumann eigenvalue of $S$ that is not Courant-sharp, the corresponding Robin eigenvalue with $h<0$, $\vert h \vert$ sufficiently small, is not Courant-sharp. Hence, for $h<0$, $\vert h \vert$ sufficiently small, the Courant-sharp Robin eigenvalues of the Laplacian on $S$ with parameter $h$ are the same as for $h=0$.

In order to prove the results of \cite{GHRS,GHRS2}, we made use of the fact that when $h>0$, the Robin eigenvalues interpolate between the Neumann eigenvalues ($h=0$) and the Dirichlet eigenvalues $(h=+\infty)$.
In addition,  in \cite{GHRS} we employed the analogue of the Faber--Krahn inequality for the Robin eigenvalues (due to Bossel--Daners) which asserts that among all bounded, Lipschitz domains $\Omega \subset \R^n$ of prescribed volume, the ball minimises $\lambda_{1,h}(\Omega)$.

In the case where $h<0$,  the Faber--Krahn inequality can be applied for the nodal domains whose boundaries intersect the boundary of $S$ in at most finitely many points, but not for the nodal domains whose boundaries intersect the boundary of $S$ in a non-trivial arc. Indeed, for bounded, planar domains with $C^2$ boundary, the Robin analogue of the Faber--Krahn inequality is reversed for $h<0$ with $\vert h \vert$ sufficiently small, \cite{FK}, and cannot be used in our analysis.

In addition, for $h<0$ we no longer have  an analogue of  the aforementioned Dirichlet--Neumann bracketing for the Robin eigenvalues as some of the Robin eigenvalues tend to $-\infty$ as $h \to -\infty$ (see Section \ref{s4}).
The latter feature of the case $h<0$ also gives an added complication that the positive Robin eigenvalues of $S$ could have multiplicity larger than $2$. We recall that to prove that the $k$-th eigenvalue is not Courant-sharp, we must show that it has no corresponding eigenfunction with $k$ nodal domains. As we do not know how to treat the case with multiplicity larger than $2$, we focus our attention on the negative Robin eigenvalues and prove the following theorem.

\begin{thm}\label{thm:negnotCS}
   There exists $h^* <0$ such that for $h < h^*$, the eigenvalues $\lambda_{k,h}(S)$, $k \geq 6$, of the Robin Laplacian on $S$ with parameter $h<0$ that are negative are not Courant-sharp.
\end{thm}

As outlined above, the methods used in our previous work \cite{GHRS,GHRS2} do not apply to the case where $h<0$.
To treat this case, we analyse the nodal sets of the Robin eigenfunctions more explicitly (in particular, the critical points and the boundary points) and use Euler's formula, Sturm's theorem and symmetry considerations to estimate the number of their nodal domains.

 It was shown by L\'ena \cite{cL16} that the analogue of Pleijel's theorem holds for the Robin Laplacian with parameter $h \geq 0$ on a bounded, connected, open subset of $\R^n$, $n \geq 2$, with a $C^{1,1}$ boundary; that is, there are finitely many Courant-sharp eigenvalues. To the best of our knowledge, the general case where $h<0$ remains open. In Section \ref{s:pos}, we prove that the number of Courant-sharp Robin eigenvalues of $S$ can be bounded from above independently of the parameter $h<0$. Hence the Robin Laplacian on the square with parameter $h<0$ has finitely many Courant-sharp eigenvalues.
\\

\textbf{Organisation of the paper.}
 In Section \ref{s4} we recall the formulae for the eigenvalues and eigenfunctions of the Robin Laplacian of an interval with parameter $h<0$. We also discuss the asymptotic behaviour of the Robin eigenvalues of the interval as $h \to -\infty$. In Section~\ref{s3} we  give an improvement of Sturm's theorem in a special case and recall Euler's formula and the symmetry properties of the Robin eigenfunctions.
 In Section \ref{sec:s5} we investigate the potential intersections of the Robin eigencurves for the square $S$. In Section \ref{s:fifth} we treat the  fourth, fifth and ninth Robin eigenvalues of $S$. We then turn our attention to the negative Robin eigenvalues of the Laplacian on $S$ in Sections \ref{s:7} and \ref{ss:0qeven}. In Section \ref{s:pos}, we obtain a uniform upper bound for the number of Courant-sharp Robin eigenvalues of the Laplacian on $S$ with parameter $h<0$.\\

\textbf{Acknowledgements.} We would like  to thank Pierre B\'erard for useful discussions about Sturm's theorem and Richard Laugesen for helpful comments.

\section{Eigenvalues and eigenfunctions of the Robin Laplacian on the square for $h<0$.}\label{s4}
\subsection{Robin eigenfunctions of an interval for $h < 0$}\label{ss:2.1}
We derive the Robin eigenfunctions of an interval for $h < 0$ (see also Section V of \cite{RL}).
We wish to solve the following problem:
\begin{align}\label{eq:1}
  -u''(x) &= \lambda u(x), \quad x \in (-\pi/2,\pi/2)\,, \notag\\
  -u'(-\pi/2) &+ hu(-\pi/2) = 0\,, \notag\\
  u'(\pi/2) &+ hu(\pi/2) = 0\,,
\end{align}
where $\lambda \in \R$ and $h < 0\,$.

For the even eigenfunctions, we have
\begin{equation*}
  u(x) = A \cos (\sqrt{\lambda}x),
\end{equation*}
where $A \in \R$ is a constant, and the boundary condition \eqref{eq:1} gives
\begin{equation}\label{eq:1a}
  \frac{\sqrt{\lambda}\pi}{2} \tan \left(\frac{\sqrt{\lambda}\pi}{2}\right) = \frac{h \pi}{2}\,.
\end{equation}

For the odd eigenfunctions, we have
\begin{equation*}
  u(x) = B \sin (\sqrt{\lambda}x),
\end{equation*}
where $B \in \R$ is a constant, and the boundary condition \eqref{eq:1} gives
\begin{equation}\label{eq:1b}
  -\frac{\sqrt{\lambda}\pi}{2} \cot \left(\frac{\sqrt{\lambda}\pi}{2}\right) = \frac{h \pi}{2}\,.
\end{equation}
By the minimax characterisation, for $h < 0$,
\begin{equation*}
  \lambda_{1,h}((-\pi/2,\pi/2)) < \mu_1((-\pi/2,\pi/2)) = 0\,,
\end{equation*}
so we must also consider the case where $\lambda < 0$.

Let $\lambda = - \pi^{-2} \beta^2$ where $\beta > 0$ is a function of $h < 0$.
Then for the even case, \eqref{eq:1a} becomes
\begin{equation}\label{eq:1c}
  \frac{\beta}{2} \tanh \left( \frac{\beta}{2} \right) = - \frac{h \pi}{2},
\end{equation}
and for the odd case, \eqref{eq:1b} becomes
\begin{equation}\label{eq:1d}
  \frac{\beta}{2} \coth \left( \frac{\beta}{2} \right) = - \frac{h \pi}{2}.
\end{equation}

We observe that $x \mapsto x \tanh (x)$ is increasing for $x \geq 0$. So for each $h < 0$, \eqref{eq:1c} has a unique solution $\beta_0(h) > 0$.

In addition, $x \mapsto x \coth (x)$ is increasing for $x \geq 0$ and $\lim_{x \to 0} x \coth (x) = 1$. So \eqref{eq:1d} has a unique solution $\beta_1(h) > 0$ for any $h$ such that $-\frac{h \pi}{2} > 1$, that is for $h < -\frac{2}{\pi}$.

Hence, when $h < -\frac{2}{\pi}$, $(-\frac{\pi}{2},\frac{\pi}{2})$ has two negative Robin eigenvalues $-\pi^{-2} \beta_0^2$ and $-\pi^{-2} \beta_1^2$.

Moreover, when $ 0 < h < -\frac{2}{\pi}$, $(-\frac{\pi}{2},\frac{\pi}{2})$ has one negative Robin eigenvalue $-\pi^{-2} \beta_0^2$.

In the case where $\lambda \geq 0$, let $\lambda = \pi^{-2} \alpha^2$ where $\alpha > 0$ is a function of $h < 0$.
Then via \eqref{eq:1a} and \eqref{eq:1b}, we obtain
\begin{equation}\label{eq:alphevn}
  \alpha \tan \left(\frac{\alpha}{2}\right) = h \pi,
\end{equation}
in the even case, and
\begin{equation}\label{eq:alphaodd}
  -\alpha \cot \left(\frac{\alpha}{2}\right) = h \pi,
\end{equation}
in the odd case, as in \cite{GHRS} equations $(2.8)$ and $(2.9)$.

We note that for $p \geq 2$, $\alpha_p = \alpha_p(h)$ is the unique non-zero solution in $[(p-1)\pi, p\pi)$ of
\begin{equation}\label{eq:R}
\frac{2 \alpha_p}{h \pi} \cos \alpha_p  + \left(1 -\frac{(\alpha_p)^2}{h^2\pi^2}\right) \sin \alpha_p  =0\, .
\end{equation}

For $-\frac{2}{\pi} < h < 0$, $\alpha_1(h)$ is the unique non-zero solution of \eqref{eq:R} in $[0,\pi)$.
For $h = -\frac{2}{\pi}$, $\alpha_1(h) = 0$.

We conclude that the Robin eigenvalues of the Robin realisation of the Laplacian with $h <0$ on $I = (-\frac{\pi}{2},\frac{\pi}{2})$ are thus given by
\begin{align*}
  \lambda_{1,h}(I) & = -\pi^{-2} \beta_0^2, \\
  \lambda_{2,h}(I) & = \begin{cases}
                         \pi^{-2} \alpha_1^2, & -\frac{2}{\pi} < h < 0, \\
                         0, & h = -\frac{2}{\pi}, \\
                         -\pi^{-2} \beta_1^2, & h < -\frac{2}{\pi},
                       \end{cases} \\
  \lambda_{p+1,h} & = \pi^{-2} \alpha_p^2, \mbox{ for } p \geq 2,
\end{align*}
with corresponding eigenfunctions
\begin{align*}
  u_{1,h}(x) & = \frac{1}{\sinh \frac {\beta_0} 2} \, \cosh \left( \frac{\beta_0 x}{\pi}\right), \\
  u_{2,h}(x) & = \begin{cases}
                         \frac{1}{\cos \frac {\alpha_1} 2} \,  \sin \left( \frac{\alpha_1 x}{\pi}\right), & -\frac{2}{\pi} < h < 0\,, \\
                         -x, & h = -\frac{2}{\pi}, \\
                         \frac{1}{\cosh \frac {\beta_1} 2} \,  \sinh \left( \frac{\beta_1 x}{\pi}\right), & h < -\frac{2}{\pi},
                       \end{cases} \\
  u_{p+1,h} & = \begin{cases}
                  \frac{1}{\sin \frac {\alpha_p} 2} \, \cos \left( \frac{\alpha_p x}{\pi}\right), & \mbox{if $p \geq 2$ is even}, \\
                  \frac{1}{\cos \frac {\alpha_p} 2} \,  \sin \left( \frac{\alpha_p x}{\pi}\right), & \mbox{if $p > 2$ is odd}.
                \end{cases}
\end{align*}

\subsection{Asymptotic formulae for the $\beta$'s and $\alpha$'s}

In Figure \ref{fig:1exp}, we plot $\beta_0(h)$, $\beta_1(h)$, $\alpha_1(h)$, $\alpha_2(h)$, $\alpha_3(h)$, $\alpha_4(h)$, $\alpha_5(h)$.

As $h\rightarrow -\infty$, we have
 \begin{equation}\label{eq:beta0}
\beta_0(h) +h\pi  \sim  - h \pi \exp h \pi \,,
\end{equation}
see Lemma IV.4 of \cite{RL} for example.
The corresponding eigenvalue $-\beta_0(h) ^2/\pi^2$ behaves like $- h^2$ as $h\rightarrow -\infty$.

In addition, as $h\rightarrow -\infty$, we have
 \begin{equation}\label{eq:beta1}
\beta_1(h)  + h\pi \sim h \pi \exp h \pi \,,
\end{equation}
see Lemma IV.4 of \cite{RL}. The corresponding eigenvalue  $-\beta_1(h) ^2/\pi^2$ behaves like $- h^2$ as $h\rightarrow -\infty$.

The corresponding eigenvalues are exponentially close as $h\rightarrow -\infty$ (see Figure \ref{fig:1exp}). We have indeed
  \begin{equation}
\beta_0 (h)- \beta_1(h)  \sim -2  h \pi \exp h \pi \,.
\end{equation}

\begin{figure}[!ht]
 \begin{center}
\includegraphics[width=10cm]{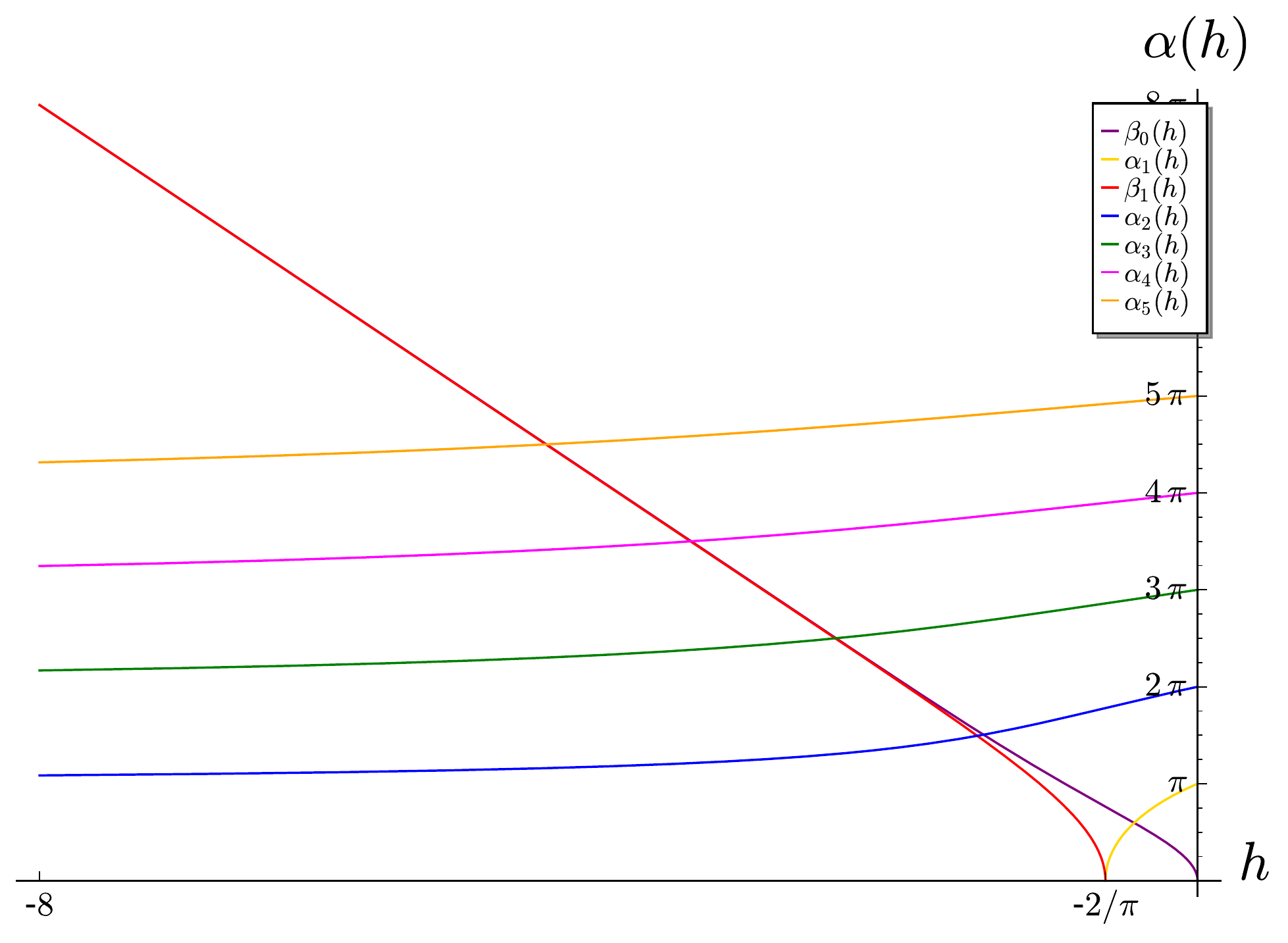}
 \caption{The graphs of $\beta_0(h)$, $\beta_1(h)$, $\alpha_1(h)$, $\alpha_2(h)$, $\alpha_3(h)$, $\alpha_4(h)$, $\alpha_5(h)$ for $-8 \leq h < 0$.}
 \label{fig:1exp}
  \end{center}
 \end{figure}

We remark that since $x \mapsto x \tan (x)$ and $x \mapsto -x \cot(x)$ are increasing functions, their inverses are also increasing (by the chain rule, for example). So for $p \geq 2$, $h \mapsto \alpha_p(h)$ is an increasing function for $h \in \R$ and we recall that
\begin{equation}\label{encada}
\alpha_p(0) = p \pi\,.
\end{equation}
See also \cite{BFK} or  \cite{RL}.
The next lemma gives the asymptotics as $h\rightarrow -\infty$ (see \cite{BFK}).
\begin{lem}
For $p\geq 1$, we have
\begin{equation}\label{encad}
\alpha_{p+1}(h) = p \pi - 2 p h^{-1} + 4 p \pi^{-1} h^{-2} -\left(8p\pi^{-2} +\frac{14}{3} p^3 \right) h^{-3}
+ \mathcal O (h^{-4})\,.
\end{equation}
\end{lem}

\begin{proof}~\\
There is a complete expansion for $\alpha_{p+1}(h)$ in powers of $\frac 1 h$.  We show how we can compute the first  four terms.
Let $\alpha = \alpha_{p+1}$. From \eqref{eq:R}, we have
\begin{equation*}
  h^2\pi^2 \tan \alpha + 2 \pi h \alpha -\alpha^2 \tan \alpha =0.
\end{equation*}
We now write $\alpha = p\pi + \mu$, and we obtain
$$
h^2\pi^2 \tan \mu + 2\pi h (p\pi + \mu) - (p\pi + \mu)^2 \tan \mu =0\,.
$$
Using that
$$
\tan \mu = \mu -\frac 13 \mu^3 + \mathcal O (\mu ^4)\,,
$$
we have that
$$
h^2\pi^2 \left(\mu -\frac 13 \mu^3\right) + 2\pi h (p\pi + \mu) - (p\pi + \mu)^2 \left( \mu-\frac 13 \mu^3\right) =\mathcal O (\mu^4)\,.
$$
By writing
$$
 \mu = \mu_ 1h^{-1} + \mu_2 h^{-2} + \mu_3 h^{-3} +\mathcal O (h^{-4})\,,
$$
we first obtain
$$
h\pi^2 \left( \mu -\frac 13 \mu^3\right) + 2\pi (p\pi + \mu) - (p\pi + \mu)^2 \left( \mu-\frac 13 \mu^3\right) h^{-1} =\mathcal O (h^{-3})\,.
$$
Then
$$
 \mu_1 \pi^2 - \frac 13 \mu_1^3 \pi^2 h^{-2}+ \mu_2 \pi^2 h^{-1} +\mu_3 \pi^2 h^{-2}  + 2\pi (p\pi + \mu_ 1h^{-1} + \mu_2 h^{-2} ) - (p\pi)^2  \mu_ 1h^{-2}  =\mathcal O (h^{-3})\,.
$$
Identifying the coefficients of the powers of $h^{-1}$ we get
$$
\begin{array}{l}
\mu_1 =-2 p \\
\mu_2= \frac{4p}{\pi} \\
\mu_3= \frac 13  \mu_1^3\pi^2 -2\pi \mu_2 + p^2 \pi^2 \mu_1= -\frac{8p}{\pi^3} -\frac{14}{3} p^3 \,,\\
\end{array}
$$
and this gives \eqref{encad}.
\end{proof}

In Figure \ref{fig:interval} we plot the first six Robin eigenvalues of the interval $(-\pi/2,\pi/2)$ for $-8 < h < 0$.
We also plot the horizontal asymptotes corresponding to the first four Dirichlet eigenvalues of this interval.
We see that as $h \to -\infty$, the eigenvalues corresponding to $\lambda_{k,h}$, $k \geq 3$, converge to the Dirichlet eigenvalues as $h \to -\infty$, while the eigenvalues $\lambda_{1,h}$, $\lambda_{2,h}$ tend to $-\infty$ as $h \to -\infty$, see \cite{BFK} for example.

\begin{figure}[!ht]
 \begin{center}
\includegraphics[width=0.8\textwidth]{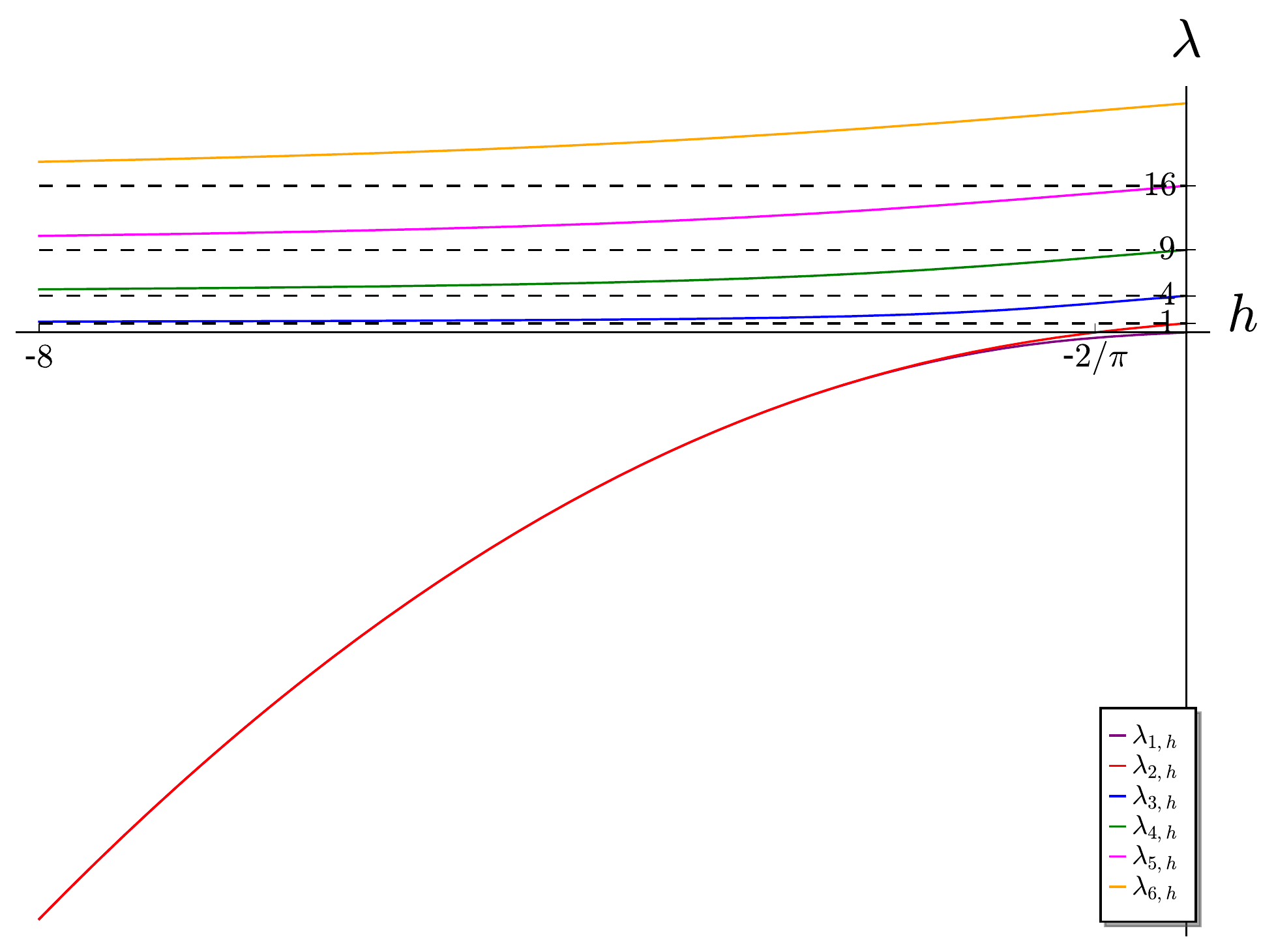}
 \caption{The graphs of the first six Robin eigenvalues of the interval $(-\pi/2,\pi/2)$ for $-8 < h < 0\,$.}
 \label{fig:interval}
  \end{center}
 \end{figure}

\subsection{Robin eigenfunctions of a square for $h < 0$\,.}

For $S = (-\frac{\pi}{2},\frac{\pi}{2})^2$, an orthogonal basis of eigenfunctions for the Robin realisation of the Laplacian on $S$ is given by
\begin{equation}\label{eq:Refnc}
u_{p,q,h} (x,y) = u_{p,h}(x) u_{q,h}(y),
\end{equation}
where, for  $p,q\in \mathbb N^*$ (where $\mathbb N^*$ is the set of the positive integers), $u_{p,h}$ is the $p$-th eigenfunction of the Robin problem in $(-\frac{\pi}{2},\frac{\pi}{2})$ as defined above.

For $-\frac{2}{\pi} < h < 0$, the Robin eigenvalues have the form
\begin{equation}\label{eq:1e}
  -2\pi^{-2}\beta_0^2, \mbox{ and } -\pi^{-2}\beta_0^2 + \pi^{-2}\alpha_q^2, \, q \in \N^{*}. \
\end{equation}

For $h \leq -\frac{2}{\pi}$, the Robin eigenvalues have the form
\begin{equation}\label{eq:1f}
  -\pi^{-2} \beta_i^2 - \pi^{-2} \beta_j^2, \, -\pi^{-2} \beta_i^2 + \pi^{-2} \alpha_q^2, \mbox{ and } \pi^{-2} \alpha_p^2 + \pi^{-2} \alpha_q^2\,,
\end{equation}
for $i, j \in \{0,1\}$ and $p, q \in \N^*$.

\section{General properties.}\label{s3}
In this section, we discuss the main tools that we will use to analyse the structure of the nodal sets of the Robin eigenfunctions. Namely, an improvement of Sturm's theorem, Euler's formula and symmetry properties.

\subsection{Improvement of Sturm's Theorem in a special case.}
In this subsection we obtain an improvement of Sturm's theorem (see, for example, \cite{BH2}) for the case of a finite linear combination of Robin eigenfunctions $u_1, u_2, \dots, u_m$ on an interval $[a,b]$.
We start by recalling  the following statement of Sturm (see \cite{BH2} for details).
\begin{thm}\label{T-st2}
Let $h\geq 0$. For a given $n>0$, let
$$
\Phi =\sum_{\ell=1}^n  c_\ell \, u_\ell
$$
$c_i \in \R$, a linear combination of the first $n$ Robin eigenfunctions in $(a,b)$ and let $a_i$ be the zeros of $\Phi$ in $(a,b)$. Then
\begin{equation}\label{Stprec}
\frac 12 \nu(a) + \sum_i \nu(a_i) +\frac 12 \nu (b) \leq n -1\,
\end{equation}
where $\nu (x)$ denotes the order of the zero at $x$.
\end{thm}%
The proof given by Sturm explicitly uses the non-negativity of $h$ and excludes the Dirichlet case.
The proof by Liouville is true for any $h\in (-\infty,+\infty]$ but is established in the less precise but more standard form:
\begin{equation}\label{Stw}
\sum_i \nu(a_i)  \leq n -1\,.
\end{equation}
 We were unable to find a reference for \eqref{Stprec} in the case $h<0\,$.
 Hence we give here a proof obtained under stronger assumptions which will be satisfied in our particular situations.

\begin{prop}\label{p:Sturmimp}
Let $h\in \mathbb R$. With the notation of Theorem \ref{T-st2}, assume that
\begin{itemize}
\item $\Phi'(a)=0\,$,
\item $\Phi''(a) <  0\,$,
\item $\Phi$ has only simple zeros in $(a,b)$ or zeros $a_j$ of multiplicity $2$ with $\Phi''(a_j) <0$\,.
\end{itemize}
Then $\Phi$ has at most $(n-2)$ zeros in $(a,b)$ (counted with multiplicities).
\noindent
If in addition $\Phi'(b)=0$ and $\Phi''(b)<  0$,  then $\Phi$ has at most $(n-3)$ zeros in $(a,b)$  (counted with multiplicities).
\end{prop}

\begin{proof}
Since $u_1$ is the first eigenfunction, it does not change sign on $(a,b)$.
Without loss of generality, we assume that $u_1 >0$ on $(a,b)$.
For $\epsilon > 0$, we consider
$$
\Phi_\epsilon:= \Phi + \epsilon u_1.
$$
We note that $u'_1(a) \neq 0$  (via the Robin boundary condition as $u_1(a) \neq 0$).
We apply the usual version of Sturm's theorem to $\Phi_\epsilon$, hence $\Phi_\epsilon$ has at most $(n-1)$ zeros in $(a,b)$.

We note that for $\epsilon$ small the simple zeros are close to the previous ones.
The double zeros are split and become essentially $a_i\pm c_i \sqrt{\epsilon}\,$.
But we also note that we have created a new zero in $(a,b)$ (behaving like $a+ c \epsilon$, with $c>0$).
Hence there are at most $(n-2)$ zeros of $\Phi$ in $(a,b)$ (counted with multiplicity).
\end{proof}

\subsection{Euler's formula.}
In this subsection we recall Euler's formula with boundary for the Robin realisation of the Laplacian with $h \in \R$ (see \cite{GHRS}).
 We note that Theorem A.1 from Appendix A of \cite{GHRS} also holds when $h<0\,$.

\begin{prop}\label{chapBH.Euler}
Let $\Omega$ be an open set in $\R^2$ with $C^{2,+}$ boundary, $u$ a Robin eigenfunction with $k$ nodal domains, $N(u)$ its zero-set.
 Let $b_0$ be the number of components of $\partial \Omega$ and $b_1$ be the number of components of $N(u) \cup\partial \Omega$. Denote by $\nu({\bf x}_i)$ and $\rho({\bf y}_i)$ the numbers of curves ending at  critical point ${\bf x}_i\in N(u)$, respectively ${\bf y}_i \in N(u)\cap \partial \Omega$. Then
\begin{equation}\label{chapBH.Emu}
k= 1 + b_1-b_0+\sum_{{\bf x}_i}\Big(\frac{\nu({\bf x}_i)}{2}-1\Big)
+\frac{1}{2}\sum_{{\bf y}_i}\rho({\bf y}_i)\,.
\end{equation}
\end{prop}

\begin{rem}\label{rem:Eul}
We remark that, the nodal set of an eigenfunction corresponding to a negative Robin eigenvalue cannot contain an immersed circle that does not intersect $\partial \Omega$.
Indeed, if the nodal set of $u$ contains an immersed circle $C$ that does not intersect $\partial \Omega$, then $u$ restricted to the nodal domain contained in $C$ and with boundary $C$ is the first eigenfunction of the Dirichlet Laplacian on this domain. So the corresponding Dirichlet eigenvalue would be negative which is a contradiction.
Therefore, if $\Omega$ is connected, we observe that for the Robin eigenfunctions corresponding to negative eigenvalues,  we always have
\begin{equation}\label{b0b1}
b_0 = b_1\,.
\end{equation} In fact, we have a stronger property: for a Robin eigenfunction corresponding to a negative eigenvalue, the closure of any nodal domain must intersect $\partial \Omega$ in at least a non-trivial arc. \end{rem}

\subsection{Symmetry of Robin eigenfunctions.}
We now recall the symmetry properties of the Robin eigenfunctions from \cite{GHRS2}.
From the formulae given in Subsection \ref{ss:2.1}, we see that the Robin eigenfunctions $u_{p,h}$, $p \in \N^*$,
of the Laplacian on $(-\frac{\pi}{2},\frac{\pi}{2})$ with parameter $h<0$ are alternately symmetric and antisymmetric:
\begin{equation*}
u_{p,h} (-x) = (-1)^{p+1} u_{p,h}(x)\, ,
\end{equation*}
like in the Dirichlet and Neumann cases.

We now consider the symmetry properties of a general eigenfunction associated with
an eigenvalue $\lambda_{p,h}(S)$ of $(-\frac{\pi}{2},\frac{\pi}{2})^2$ which reads:
\begin{equation*}
   u(x,y) = \sum_{i,j \in \N : \lambda_{n,h}(S) = \pi^{-2}(\alpha_{i}^2 + \alpha_{j}^2)}
    a_{ij}\,  u_{i+1}(x)u_{j+1}(y)\,,
 \end{equation*}
 where $ u_{p}$ (or $ u_{p,h}$ if we want to include the reference to the Robin parameter) is the $ p$--th eigenfunction of the Robin Laplacian on $ (-\frac{\pi}{2},\frac{\pi}{2})$ with parameter $h<0$, and where we recall the convention that $\alpha_0=i \beta_0$, and $\alpha_1=i\beta_1$ (if $h< -\frac 2 \pi$).

 By considering the transformation $ (x,y) \mapsto (-x,-y)$, we obtain
 \begin{equation}\label{eq:2.9b}
   u(-x,- y) = \sum_{i,j \in \N : \lambda_{p,h}(S) = \pi^{-2}(\alpha_{i}^2 + \alpha_{j}^2)}
    a_{ij}\, (-1)^{i+j}  u_{i+1}(x)u_{j+1}(y)\,.
 \end{equation}
 \begin{rem}\label{remsym}
 If $(i+j)$ is odd for any pair $(i,j)$ such that $\lambda_{p,h}(S) = \pi^{-2}(\alpha_{i}^2 + \alpha_{j}^2)$, then we get by \eqref{eq:2.9b}, $ u(-x,-y) = - u(x,y)$. Hence $u$ has an even number of nodal domains (see also Remark 2.2 of \cite{GHRS}).
 \end{rem}

\section{Analysis of crossings.}\label{sec:s5}

In this section, we study the potential number of intersections between eigencurves corresponding to distinct pairs. We reconsider the arguments of the proof of Proposition 7.1 from \cite{GHRS} and show that they also hold for the case when $h < 0$. We deduce the corresponding result to Proposition 7.1 from \cite{GHRS} in certain cases.

Suppose that $\lambda_{p,q,h}(S) = \lambda_{p',q',h}(S)$ for some $h=h_0$.
Without loss of generality, suppose that $p<p' \leq q'<q$.
We are interested in the other potential crossings between the curves $h \mapsto \lambda_{p,q,h}(S)$ and $h \mapsto \lambda_{p',q',h}(S)$, so we consider the function
\begin{equation}
(0,+\infty)\ni h\mapsto \sigma(h):=\frac{1}{\pi^2}\left( \alpha_{p}(h)^2 + \alpha_{q} (h)^2 - \alpha_{p'}(h)^2
-\alpha_{q'}(h)^2\right).
\end{equation}

The zeros of $\sigma$ correspond to the values of $h$ for which the curves corresponding to $(p,q)$, $(p',q')$
intersect. We note that
\begin{equation}\label{eq:3.1a}
\sigma'(h) = \frac{2}{\pi^2}\left( \alpha_{p}(h)\alpha'_{p}(h)  + \alpha_{q} (h)\alpha'_{q}(h) - \alpha_{p'}(h)\alpha'_{p'}(h) -\alpha_{q'}(h)\alpha'_{q'}(h) \right)\,.
\end{equation}

\begin{prop}\label{p:sign}
For distinct pairs $(p,q)$ and $(p',q')$, with $p\leq q$ and $p'\leq q'$, the sign of $\sigma'(h)$ at a zero $h$ of $\sigma$ is given as in the last two columns of the following table.
\begin{table}[h!]
	\centering
	\caption{The sign of $\sigma'$ at a zero of $\sigma$.\label{tab:sign}}
	\begin{tabular}{|c|c|c|c|}
	\hline
	 Case & $p, q, p', q'$ & $-\frac{2}{\pi} < h < 0$ & $h < -\frac{2}{\pi}$ \\
	\hline
    (i) & $p=0, q \geq 2; p'=q'=1$ & $<0$ & $< 0$ \\
    \hline
    (ii) & $p=0, q \geq 3; p'=1,q' \geq 2$ & $ < 0$ & $\text{sign}((a_0+a_q)(a_0a_q - a_1a_{q'}))$ \\
    \hline
    (iii) & $p=0, q \geq 3; p', q' \geq 2$ & $< 0$ & $> 0$\\
    \hline
    (iv) & $p=1, q \geq 3; p',q' \geq 2$ & $ <0$ & $> 0$ \\
    \hline
    (v) & $p \geq 2, q \geq 4; p', q' \geq 3$ & $ <0 $ & $ > 0$
    \\
    \hline
    \end{tabular}
\end{table}
\end{prop}

\begin{proof}
We deduce from the formulas that determine $\alpha_k$ (see \eqref{eq:1c}, \eqref{eq:1d} and \cite{GHRS}) that $h\mapsto \alpha_k(h)$ satisfies the differential equation
\begin{equation}
\frac{\alpha'_k}{\alpha_k} \left( h\pi + \frac{\alpha_k^2}{2} + \frac{h^2 \pi^2}{2} \right)=\pi\, ,
\end{equation}
which implies
\begin{equation}\label{eq:3.2}
\alpha'_k \alpha_k \left( h\pi + \frac{\alpha_k^2}{2} + \frac{h^2 \pi^2}{2} \right)=\pi  \alpha_k^2\, .
\end{equation}
We remark that \eqref{eq:3.2} is true for any $h \in \R$ where we use the convention that $\alpha_0(h)=i \beta_0 (h)$ for $h<0$ and $\alpha_1(h) = i \beta_1(h)$ for $h<-\frac{2}{\pi}$.
For $h \in \R$ and $k \in \N$, we define
\begin{equation}\label{eq:3.2a}
a_k (h) =  h\pi + \frac{\alpha_k^2}{2} + \frac{h^2 \pi^2}{2} \, .
\end{equation}
\noindent
We note that $h\pi + \frac{h^2\pi^2}{2} \geq 0$ if and only if $h \leq -\frac{2}{\pi}$ so clearly $a_k(h) \geq 0$ for $k\geq 2$ and $h \leq -\frac{2}{\pi}$.\\
On the other hand, by monotonicity with respect to $h$ of the $k$-th eigenvalue, we have that for $k \in \N$,
$$\alpha_k \alpha_k' \geq 0\,.$$
Together with formula \eqref{eq:3.2}, this implies:
\begin{equation}\label{eq:ak2}
 a_k(h) \geq 0 \,,\, \mbox{ for }  k\geq 2 \mbox{ and } h < 0\,.
 \end{equation}
  The same argument holds for $k=1$ and $h\in (-\frac{2}{\pi}, 0)$, hence:
  \begin{equation}\label{eq:a1+}
  a_1(h) \geq 0 \mbox{ for } h\in (-2/\pi, 0)\,.
  \end{equation}
  Moreover, \eqref{eq:3.2} also shows that
 \begin{equation} \label{eq:a1-}
  a_1(h) < 0 \mbox{ for } h\in (-\infty, -2/\pi)\,,
  \end{equation}
 and that
 \begin{equation}\label{eq:a0}
 a_0(h)< 0 \mbox{ for } h<0\,.
 \end{equation}

We now analyse $\sigma'(h)$. We have
$$
\sigma'(h) =  \frac{2}{\pi} \left( \frac{\alpha_{p}^2}{ a_{p}}+ \frac{\alpha_{q}^2}{ a_{q}} - \frac{\alpha_{p'}^2}{ a_{p'}} - \frac{\alpha_{q'}^2}{ a_{q'}} \right)
=  - \frac{4}{\pi} \left(h\pi +   \frac{h^2 \pi^2}{2} \right)\left( \frac{1}{a_{p}} + \frac{1}{a_{q}}  -  \frac{1}{a_{p'}} - \frac{1}{a_{q'}} \right) \,.
$$
If we assume that $\sigma (h)=0$, which implies $$ a_{p} + a_{q} = a_{p'} + a_{q'}\,,$$
then at the crossing points we obtain:
\begin{equation}
\sigma'(h) =  - \frac{4}{\pi} \pi h  \left(1 +   \frac{h \pi}{2} \right)\left ( \frac{(a_{p}+a_{q})(a_{p'}a_{q'}-a_{p}a_{q})}{(a_{p}a_{q}a_{p'}a_{q'})}\right)\,.
\end{equation}
We now  deduce the sign of $\sigma'(h)$ as given in Table \ref{tab:sign}.
\begin{itemize}
\item For case (i), we write
$$
\sigma'(h) =  - \frac{4}{\pi} \pi h  \left(1 +   \frac{h \pi}{2} \right)\left ( \frac{ 2  a_{1} (a_{1}^2 -a_{0}a_{q})}{a_{0}a_{q}a_{1}^2 }\right)\, ,
$$
and can use \eqref{eq:ak2}--\eqref{eq:a0}.
\item For case (ii), we write
$$
\sigma'(h) =  - \frac{4}{\pi} \pi h  \left(1 +   \frac{h \pi}{2} \right)\left ( \frac{(a_{1}+a_{q'})(a_{1}a_{q'}-a_{0}a_{q})}{(a_{0}a_{q}a_{1}a_{q'})}\right)\,.
$$
Using again \eqref{eq:ak2}--\eqref{eq:a0}, $\sigma'(h)$ has the same sign as $(a_{1}+a_{q'})(-a_{1}a_{q'}+a_{0}a_{q})$ and is negative for $h\in (-\frac 2\pi,0)$.
\item For case (iii), we write
$$
\sigma'(h) =  - \frac{4}{\pi} \pi h  \left(1 +   \frac{h \pi}{2} \right)\left ( \frac{(a_{p'}+a_{q'})(a_{p'}a_{q'}-a_{0}a_{q})}{(a_{0}a_{q}a_{p'}a_{q'})}\right)\,.
$$
This has the same sign as $- (1 +  \frac{h \pi}{2} )$.
\item
For case (iv), we write
$$
\sigma'(h) =  - \frac{4}{\pi} \pi h  \left(1 +   \frac{h \pi}{2} \right)\left ( \frac{(a_{p'}+a_{q'})(a_{p'}a_{q'}-a_{1}a_{q})}{(a_{1}a_{q}a_{p'}a_{q'})}\right)\,.
$$
This has the same sign as $(a_{p'}a_{q'}-a_{1}a_{q})$. It is negative for $h \in (-\frac 2 \pi,0)$, and it is positive for $h\in (-\infty, -\frac 2 \pi)$ if $q\geq 3$.
\item For case (v), we write
$$
\sigma'(h) =  - \frac{4}{\pi} \pi h  \left(1 +   \frac{h \pi}{2} \right)\left ( \frac{(a_{p}+a_{q})(a_{p'}a_{q'}-a_{p}a_{q})}{(a_{p}a_{q}a_{p'}a_{q'})}\right)\,.
$$
This has the same sign as $ \left(1 +   \frac{h \pi}{2} \right) (a_{p'}a_{q'}-a_{p}a_{q})$.
As $a_k(h) \geq 0$ for $k \geq 2$, we have that for $\epsilon > 0$, $a_{p'} = a_{p} + \epsilon$
and $a_{q'} = a_{q} - \epsilon$. So $a_{p'}a_{q'}-a_{p}a_{q} = \epsilon(a_q - a_p) - \epsilon^2 < 0$.
\end{itemize}
\end{proof}

For an interval in which the derivative of $\sigma$ has constant sign at all crossing points, there can be at most one point of intersection between the two curves. We thus deduce the following corollary.

\begin{cor}\label{c:crossing}
Let $(p,q)$ and $(p',q')$ be distinct pairs with $p\leq q$ and $p'\leq q'$.
Then, in case (i) of Table \ref{tab:sign}, there is at most one value of $h$ in $(-\infty,0)$ such that
$\lambda_{p,q,h}(S) = \lambda_{p',q',h}(S)$.
In cases (iii), (iv) and (v) of Table \ref{tab:sign}, there are at most two values of $h$ in $(-\infty,0)$ such that
$\lambda_{p,q,h}(S) = \lambda_{p',q',h}(S)$.
\end{cor}

\begin{rem}\label{rem:label}
  From Table \ref{tab:sign}, there is at most one value of $h \in (-\frac 2 \pi, 0)$ such that $\lambda_{p,q,h}(S) = \lambda_{p',q',h}(S)$ for distinct pairs, $p\leq q$, $p'\leq q'$. Since the labelling of the eigenvalue can only change at a crossing, we deduce that there exists $\hat{h} \in (-\frac 2 \pi, 0)$ such that the labelling of the eigenvalues $\lambda_{k,h}(S)$ for $\hat{h} \leq h < 0$ is the same as that for the Neumann case $h=0$.
\end{rem}

\noindent
To deal with case (ii) of Table \ref{tab:sign}, we make use of the following lemma.

\begin{lem}\label{lem:monotonicity}~
\begin{enumerate}
  \item[(i)] For $k > \ell \geq 2$, $h \mapsto \alpha_k(h)^2 - \alpha_{\ell}(h)^2$ is increasing for $ h \leq -\frac{2}{\pi}$.
  \item[(ii)]  For $k > \ell \geq 1$, $h \mapsto \alpha_k(h)^2 - \alpha_{\ell}(h)^2$ is decreasing for $-\frac{2}{\pi} < h < 0$.
  \item[(iii)] For $h < -\frac{2}{\pi}$, $h \mapsto \beta_0(h)^2 - \beta_1(h)^2$ is increasing.
\end{enumerate}
\end{lem}

\begin{proof}
  To prove (i), we wish to show that $\alpha_k'(h) \alpha_k(h) - \alpha_{\ell}'(h) \alpha_{\ell}(h) \geq 0$.
  We observe that
  \begin{align*}
    \alpha_k' \alpha_k & \geq \alpha_{\ell}' \alpha_{\ell} \\
    \iff \frac{\pi \alpha_k^2}{a_k} & \geq \frac{\pi \alpha_{\ell}^2}{a_{\ell}} && \text{ by \eqref{eq:3.2},} \\
    \iff \alpha_k^2 a_{\ell} & \geq \alpha_{\ell}^2 a_k && \text{ by \eqref{eq:ak2},} \\
    \iff \alpha_k^2 \left(h\pi + \frac{\alpha_{\ell}^2}{2} + \frac{h^2\pi^2}{2}\right) & \geq \alpha_{\ell}^2 \left(h\pi + \frac{\alpha_k^2}{2} + \frac{h^2\pi^2}{2}\right),
  \end{align*}
   which holds trivially for $h = -\frac{2}{\pi}$, and for $h < -\frac{2}{\pi}$, we have that it is equivalent to
  \begin{equation*}
      \alpha_k^2 \left(h\pi + \frac{h^2\pi^2}{2}\right) \geq \alpha_{\ell}^2 \left(h\pi + \frac{h^2\pi^2}{2}\right) \iff \alpha_k^2 \geq \alpha_{\ell}^2\, ,
  \end{equation*}
  which holds since $k > \ell \,$.

   The fact that $h \mapsto \alpha_k(h)^2 - \alpha_{\ell}(h)^2$ is decreasing for $-\frac{2}{\pi} < h < 0$ follows by analogous arguments since $(h\pi + \frac{h^2\pi^2}{2}) <0$ in this case.

   We note that when $k=1$, $\ell = 0$, $h < -\frac{2}{\pi}$, we have $\alpha_1(h) = i\beta_1(h)$, $\alpha_0(h) = i\beta_0(h)$ and the above arguments give rise to item (iii) by using that $\beta_0(h) \geq \beta_1(h)$.
\end{proof}

\begin{prop}\label{p:caseii}
Suppose that $p=0$, $q \geq 3$, $p'=1$ and $2 \leq q' < q$. For $h < -\frac{2}{\pi}$, we have that
$$ \lambda_{1,q',h}(S) = \pi^{-2} (-\beta_1(h)^2 + \alpha_{q'}(h)^2)
< \pi^{-2} (-\beta_0(h)^2 + \alpha_{q}(h)^2) = \lambda_{0,q,h}(S).$$
That is, for $h < -\frac{2}{\pi}$, the curve corresponding to $(1,q')$ lies below that corresponding to $(0,q)$.
\end{prop}

\begin{proof}~\\
  We observe that $\pi^{-2} (-\beta_1(h)^2 + \alpha_{q'}(h)^2) < \pi^{-2} (-\beta_0(h)^2 + \alpha_{q}(h)^2)$ if and only if \break  $\pi^{-2} ( \beta_0(h)^2 -\beta_1(h)^2 ) < \pi^{-2} (\alpha_{q}(h)^2 - \alpha_{q'}(h)^2)$.

  By Lemma \ref{lem:monotonicity},
  $$ \beta_0(h)^2 -\beta_1(h)^2 \leq \beta_0(-2/\pi)^2 -\beta_1(-2\pi)^2 \approx 5.7569,$$
  and
  $$ \alpha_{q}(h)^2 - \alpha_{q'}(h)^2 \geq \alpha_{q}(-\infty)^2 - \alpha_{q'}(-\infty)^2
  =((q-1)^2 - (q'-1)^2)\pi^2 > \pi^2.$$
  So
  $$ \beta_0(h)^2 -\beta_1(h)^2 < 5.76 < \pi^2 < \alpha_{q}(h)^2 - \alpha_{q'}(h)^2.$$
\end{proof}

\section{The fourth, fifth, ninth Robin eigenvalues for $h<0$.}\label{s:fifth}
\subsection{The fourth and fifth Robin eigenvalues for $h<0$.}
In this subsection, we prove Theorem \ref{thm:5}.
In order to work with the fifth Robin eigenvalue $\lambda_{5,h}(S)$ for $h < 0$, we must first consider which pairs it corresponds to. This leads us to consider whether the curves corresponding to the pairs $(0,2)$ and $(1,1)$ intersect for some $h<0$. In fact, by using the results of the previous section in the case $p=0$, $q=2$, $p'=q'=1$, we prove the following lemma.

\begin{lem}\label{l:1}
For $h < 0$, the fifth Robin eigenvalue of $S$ is given by the pair $(0,2)$, that is $\lambda_{5,h}(S) = \lambda_{0,2,h}(S)$.
\end{lem}

\begin{proof}
  Suppose that the curves corresponding to $\lambda_{0,2,h}(S)$ and $\lambda_{1,1,h}(S)$ intersect for some $h < 0$.
  By Proposition \ref{p:sign}, we have that in this case $\sigma'(h) < 0$ for $h<0$ (Case (i) of Table \ref{tab:sign}).

  For $-\frac{2}{\pi} < h < 0$, we have
  $$ \sigma(h) = \pi^{-2}(-\beta_0(h)^2 + \alpha_2(h)^2 - 2 \alpha_1(h)^2).$$
  So $\sigma(0) = 2$ and, numerically, $\sigma(-\frac{2}{\pi}) \approx 25.5669 > 0$.
  Hence, if there was a crossing for some $-\frac{2}{\pi} < h < 0$, then there should be at least two crossings on this interval and $\sigma'(h)$ would be positive on some subinterval. This gives a contradiction.

  For $h < -\frac{2}{\pi}$, we have
  $$ \sigma(h) = \pi^{-2}(-\beta_0(h)^2 + \alpha_2(h)^2 + 2 \beta_1(h)^2),$$
  so $\lim_{h \to -\infty} \sigma(h) = +\infty$. As $\sigma(-\frac{2}{\pi}) > 0$, we obtain a contradiction as above.

  Therefore the curves corresponding to $\lambda_{0,2,h}(S)$ and $\lambda_{1,1,h}(S)$ do not intersect each other.
  In addition, $\lambda_{0,2,0}(S) > \lambda_{1,1,0}(S)$.

  We also observe that the curves corresponding to $\lambda_{j,k,h}(S)$ with $j \geq 1$, $k \geq 2$, do not intersect the curve corresponding to $\lambda_{0,2,h}(S)$ for any $h < 0$. Indeed, for $-\frac{2}{\pi} < h < 0$,
  $$\pi^{-2}(\alpha_j^2 + \alpha_k^2) \geq \pi^{-2}\alpha_2^2 \geq \pi^{-2}(-\beta_0^2 + \alpha_2^2).$$
  For $h < -\frac{2}{\pi}$, since $\beta_0(h) \geq \beta_1(h)$, we have
  $$\pi^{-2}(\alpha_j^2 + \alpha_k^2) \geq \pi^{-2}(-\beta_1^2 + \alpha_2^2) \geq \pi^{-2}(-\beta_0^2 + \alpha_2^2).$$

  We conclude that $\lambda_{5,h}(S)$ is given by the pair $(0,2)$.
\end{proof}

\noindent
For $-\frac{2}{\pi} < h < 0$, we have that
\begin{align*}
\lambda_{1,h}(S) &= -2\pi^{-2}\beta_0(h)^2,\\
\lambda_{2,h}(S) &= -\pi^{-2}\beta_0(h)^2 + \pi^{-2}\alpha_1(h)^2 = \lambda_{3,h}(S),\\
\lambda_{4,h}(S) &= 2\pi^{-2}\alpha_1(h)^2,\\
\lambda_{5,h}(S) &= - \pi^{-2}\beta_0(h)^2 + \pi^{-2}\alpha_2(h)^2.
\end{align*}
We observe that there exists a unique  $h_2^* \in (-\frac{2}{\pi} , 0)$ such that $\lambda_{2,h}(S) \geq 0$ for $h_2^* \leq h < 0$, and $\lambda_{2,h}(S) < 0$ for $-\frac{2}{\pi} < h < h_2^*$. Numerically, we compute that $h_2^* \approx -0.4382$.

For $h < -\frac{2}{\pi}$, we have that
\begin{align*}
\lambda_{1,h}(S) &= -2\pi^{-2}\beta_0(h)^2,\\
\lambda_{2,h}(S) &= -\pi^{-2}\beta_0(h)^2 + \pi^{-2}\beta_1(h)^2 = \lambda_{3,h}(S),\\
\lambda_{4,h}(S) &= -2\pi^{-2}\beta_1(h)^2,\\
\lambda_{5,h}(S) &= - \pi^{-2}\beta_0(h)^2 + \pi^{-2}\alpha_2(h)^2.
\end{align*}
So $\lambda_{1,h}(S), \lambda_{2,h}(S)$ are Courant-sharp for all $h < 0$, but $\lambda_{3,h}(S)$ is not for any $h < 0$.

We observe that $x=0$ and $y=0$ are nodal lines of $u_{1,1}(x,y)$ for $h<0$.
We have therefore proved the result for the fourth Robin eigenvalue $\lambda_{4,h}(S) = \lambda_{1,1,h}(S)$ given in Theorem \ref{thm:5}.
We now complete the proof of Theorem \ref{thm:5} by treating the fifth Robin eigenvalue of the Laplacian on $S$.

\begin{prop}\label{p:fifth}
The fifth Robin eigenvalue $\lambda_{5,h}(S) = \lambda_{0,2,h}(S)$ of $S$ is Courant-sharp for all $h < 0$.
In particular, the corresponding eigenfunction
$$  \cosh\left(\frac{\beta_0(h) x}{\pi}\right) \cos\left(\frac{\alpha_2(h)y}{\pi}\right) +  \cosh \left(\frac{\beta_0(h)y}{\pi}\right) \cos\left(\frac{\alpha_2(h) x}{\pi}\right)\,.$$
has five nodal domains.
\end{prop}

\begin{proof}
Any eigenfunction corresponding to $\lambda_{0,2,h}(S)$ has the form
$$ \cos\theta \cosh\left(\frac{\beta_0(h) x}{\pi}\right) \cos\left(\frac{\alpha_2(h)y}{\pi}\right) +  \sin\theta \cosh \left(\frac{\beta_0(h)y}{\pi}\right) \cos\left(\frac{\alpha_2(h) x}{\pi}\right)\,.$$
In the case $\theta=\frac \pi 4$, we consider the nodal set of
$$ \tilde{u}_{0,2}(x,y) :=  \cosh\left(\frac{\beta_0(h) x}{\pi}\right) \cos\left(\frac{\alpha_2(h)y}{\pi}\right) +  \cosh \left(\frac{\beta_0(h)y}{\pi}\right) \cos\left(\frac{\alpha_2(h) x}{\pi}\right)\,.$$
We first observe that $\{x=0\}$ does not belong to the nodal set. Indeed,
$$  \tilde{u}_{0,2}(0,y) = \cos\left(\frac{\alpha_2(h)y}{\pi}\right) +  \cosh \left(\frac{\beta_0(h)y}{\pi}\right)>0 \,.$$
Similarly $\{y=0\}$ does not belong to the nodal set.

We also observe that $\tilde{u}_{0,2}(-x,y) = \tilde{u}_{0,2}(x,y)$ and $\tilde{u}_{0,2}(x,-y) = \tilde{u}_{0,2}(x,y)$.
So, by symmetry, it is sufficient to analyse the nodal set of $\tilde{u}_{0,2}(x,y)$ in $[0,\frac{\pi}{2}]^2$.

At the corner $(\frac \pi 2, \frac \pi 2)$, we have
$$\tilde{u}_{0,2}(\pi /2, \pi /2) = 2 \cosh (\beta_0(h)/2) \cos (\alpha_2(h)/2).$$
We note that $\alpha_2(h) \in (\pi, 2\pi)$ for $h < 0$, so $\cos (\alpha_2(h)/2) < 0$ for $h < 0$.
Hence $\tilde{u}_{0,2}(\frac{\pi}{2}, \frac{\pi}{2}) < 0$ for all $h < 0$, and the nodal set of $\tilde{u}_{0,2}(x,y)$ does not intersect the corner $(\frac \pi 2, \frac \pi 2)$.
By symmetry, $\tilde{u}_{0,2}(\frac{\pi}{2}, -\frac{\pi}{2}) < 0$ for all $h < 0$.

 By Sturm's theorem (see \cite{BH2} and references therein), $\tilde{u}_{0,2}(\frac{\pi}{2}, y)$ has at most $2$ zeros in $(-\frac{\pi}{2},\frac{\pi}{2})$. Since $\tilde{u}_{0,2}(\frac{\pi}{2}, \frac{\pi}{2}) < 0$, $\tilde{u}_{0,2}(\frac{\pi}{2}, 0) > 0$ and
$\tilde{u}_{0,2}(\frac{\pi}{2}, -\frac{\pi}{2}) < 0$ for all $h < 0$, $\tilde{u}_{0,2}(\frac{\pi}{2}, y)$ has exactly $2$ zeros in $(-\frac{\pi}{2},\frac{\pi}{2})$.

So the nodal set of $\tilde{u}_{0,2}(x,y)$ intersects the edge $x = \frac{\pi}{2}$ exactly once for $y \in (0,\frac{\pi}{2})$, by symmetry. Thus, the nodal set of $\tilde{u}_{0,2}(x,y)$ intersects $\partial S$ in exactly $8$ points ($2$ points on each edge of $\partial S$). Since $\{x=0\}$, $\{y=0\}$ do not belong to the nodal set, $\tilde{u}_{0,2}(x,y)$ has $5$ nodal domains (see also Figure \ref{fig1}).
\end{proof}

In Figure~\ref{fig1} below, we plot the corresponding fifth Robin eigenfunction
$$ \cos \theta \cosh\left(\frac{\beta_0(h) x}{\pi}\right) \cos\left(\frac{\alpha_2(h)y}{\pi}\right) + \sin \theta \cosh \left(\frac{\beta_0(h)y}{\pi}\right) \cos\left(\frac{\alpha_2(h) x}{\pi}\right),$$
for $(x,y) \in (-\frac{\pi}2,\frac{\pi}2)^2$, $h=-0.1$, $h=-0.6366$, $ h=-2$ respectively, and various values of $\theta$.
For $h=-0.6366$ and $ h=-2$, we see that there is more than one value of $\theta$ giving rise to $5$ nodal domains.

 \begin{figure}[htp!]
\centering
  \includegraphics[width=\textwidth]{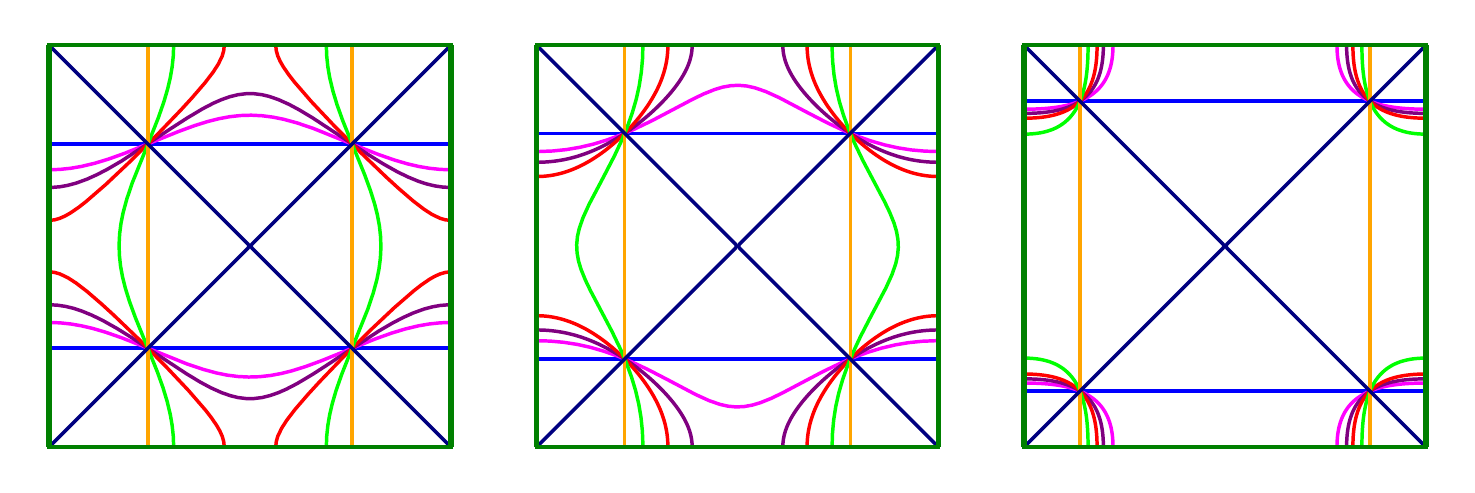}%
  \caption{ The nodal sets of the fifth Robin eigenfunction for $h=-0.1$ (left), $h=-0.6366$ (centre), $h=-2$ (right) respectively and $\theta = 0$ (blue), $\theta=\frac{\pi}8$ (magenta), $\theta=\frac{3\pi}{16}$ (purple), $\theta=\frac{\pi}4$ (red), $\theta=\frac{3\pi}8$ (lime), $\theta=\frac{\pi}2$ (orange) and $\theta=\frac{3\pi}4$ (navy).}
\label{fig1}
\end{figure}

\subsection{The ninth Robin eigenvalue for $h<0$.}\label{s:ninth}
In this subsection, we prove Proposition \ref{p:ninth}.
Numerically, we see that $\lambda_{9,h}(S)$ is either given by the pair $(2,2)$ or the pair $(0,3)$ (see Figure~\ref{fig:lambda9}).

\begin{lem}\label{lem:ninth}
  There exists $h_9^* < 0$ such that the ninth Robin eigenvalue $\lambda_{9,h}(S)$ of $S$ is given by the pair $(2,2)$ for $h_9^* \leq h < 0$, and by the pair $(0,3)$ for $h \leq h_9^*$.
\end{lem}

\begin{proof}
For $-\frac{2}{\pi} < h < 0$, case (ii) from Table \ref{tab:sign} gives $\sigma'(h) > 0$.
In addition, we have
$$ \alpha_3(-2/\pi)^2 - \beta_0(-2/\pi)^2 - \alpha_2(-2/\pi)^2 - \alpha_1(-2/\pi)^2 \approx 43.6821 > 0,$$
and
$$ \alpha_3(0)^2 - \beta_0(0)^2 - \alpha_2(0)^2 - \alpha_1(0)^2 = 4 \pi^2 > 0.$$
So the curves corresponding to $(0,3)$ and $(1,2)$ do not intersect for $-\frac{2}{\pi} < h < 0$.
In fact, the curve corresponding to $(0,3)$ lies above the curve corresponding to $(1,2)$ for $-\frac{2}{\pi} < h < 0$.
By Proposition \ref{p:caseii} with $q=3$ and $q'=2$, this is still the case  for $h < -\frac{2}{\pi}$.

We see that the curves corresponding to $\lambda_{2,2,h}(S)$ and $\lambda_{0,3,h}(S)$ must cross exactly once for some $h_9^*$, as $\lambda_{2,2,-\infty}(S) = 2$, $\lambda_{0,3,-\infty}(S) = -\infty$ while $\lambda_{2,2,0}(S)= 8$, $\lambda_{0,3,0}(S) = 9$. Numerically, we find that  $h_9^* \approx -1.6293$ (see also Figure \ref{fig:lambda9}).
\end{proof}

 \begin{figure}[htp!]
\centering
  \includegraphics[width=0.9\textwidth]{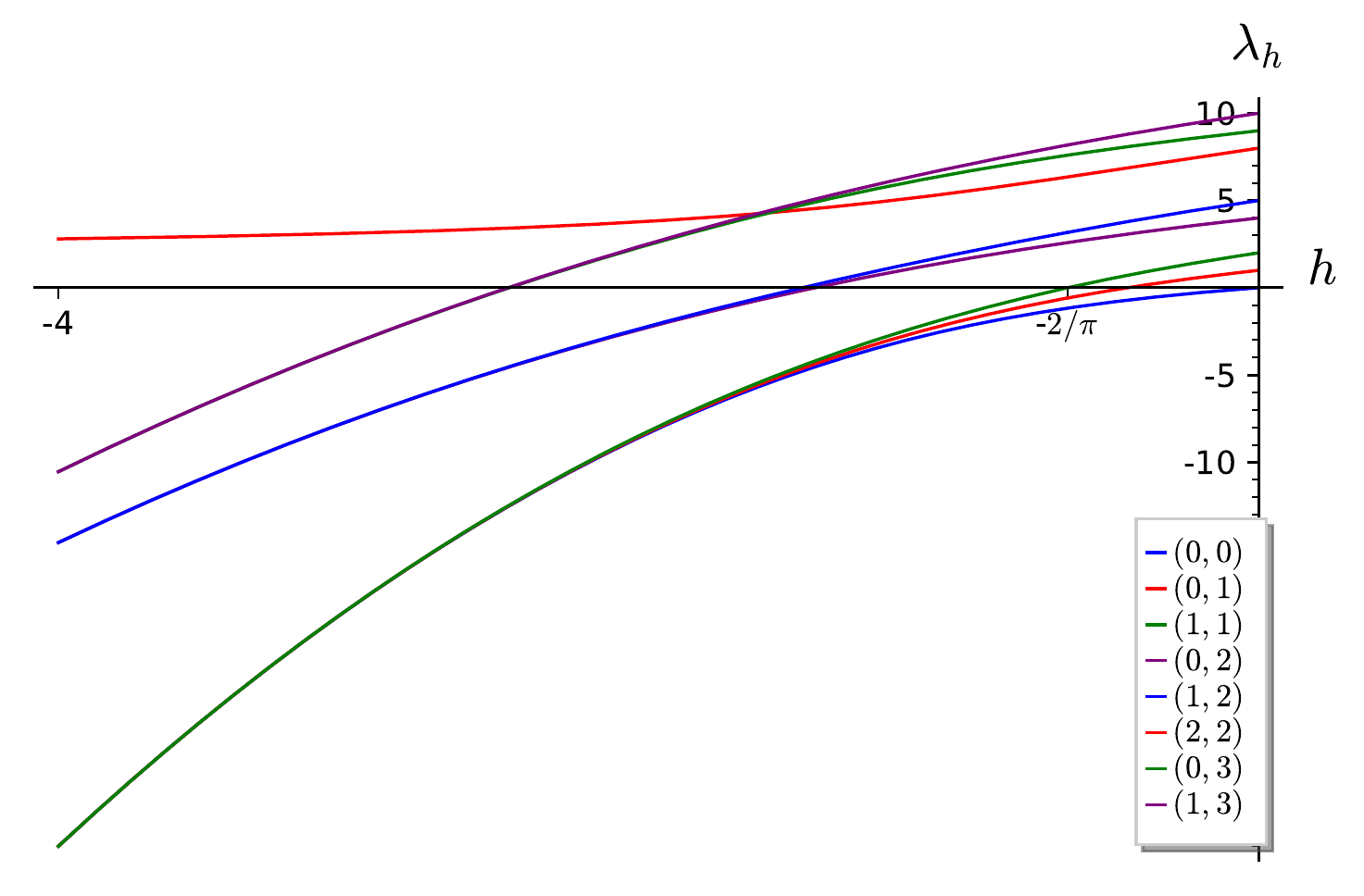}%
  \caption{The first nine Robin eigenvalues of $S$ for $-4 \leq h < 0$.}
\label{fig:lambda9}
\end{figure}

\begin{prop}\label{lem:ninth2}
  The ninth Robin eigenvalue $\lambda_{9,h}(S)$ of $S$ is Courant-sharp for $h_9^* \leq h < 0$, and is not Courant-sharp for $h < h_9^*$.
\end{prop}

\begin{proof}
For $ h_9^* \leq  h < 0$, $\lambda_{9,h}(S)$ is given by the pair $(2,2)$.
A corresponding eigenfunction\footnote{ For $h=h_9^*$ the eigenspace has dimension $3$ but this does not affect the argument.} is
$$\cos \left( \frac{\alpha_2(h) x}{\pi}\right) \cos \left( \frac{\alpha_2(h) y}{\pi}\right)$$
which has 4 nodal lines $x = \pm \frac{\pi^2}{2\alpha_2(h)}$, $y = \pm \frac{\pi^2}{2\alpha_2(h)}$
giving rise to $9$ nodal domains. Hence in this case, $\lambda_{9,h}(S)$ is Courant-sharp.

For $ h < h_9^*$, $\lambda_{9,h}(S)$ is given by the pair $(0,3)$ and no other pair. We therefore deduce that it is not Courant-sharp.
Indeed, any corresponding eigenfunction has the form
$$ \Phi_{0,3,h,\theta}(x,y) := \cos \theta   \cosh\left(\frac{\beta_0(h) x}{\pi}\right) \sin\left(\frac{\alpha_3(h)y}{\pi}\right) + \sin \theta  \cosh \left(\frac{\beta_0(h)y}{\pi}\right) \sin\left(\frac{\alpha_3(h) x}{\pi}\right)\, ,$$
and $ \Phi_{0,3,h,\theta}(-x,-y) = -\Phi_{0,3,h,\theta}(x,y)$. So $\Phi_{0,3,h,\theta}(x,y)$ has an even number of nodal domains  (see Remark \ref{remsym}).
\end{proof}

\section{Analysis of the spectrum as $h\rightarrow -\infty$.}\label{s:7}
In this section, we consider the negative Robin eigenvalues of $S$.
We observe that these eigenvalues correspond to pairs of the form $(0,q)$, $0\leq  q\leq N(h)$, and $(1,q')$, $0 \leq q' \leq N'(h)$,  where $N(h), N'(h)$ are integers depending on $h$.\\
Indeed, let $p, q \geq 2$. Then by monotonicity and \eqref{encad},  for all $h<0$, we have
$$ \lambda_{p,q,h}(S) = \pi^{-2} (\alpha_p(h)^2 + \alpha_q(h)^2) \geq (p-1)^2 + (q-1)^2 \geq 2 >0\,.$$
On the other hand, $\lambda_{0,q,h}(S)$ and  $\lambda_{1,q,h}(S)$ tend to $-\infty$ as $h \to -\infty$ since $q\pi \geq \alpha_q(h) \geq (q-1)\pi$  (for $q\geq 2$ and $h<0$)  while $\beta_0(h)$, $\beta_1(h)$ tend to $-\infty$ as $h \to -\infty$.

\subsection{Multiplicities.}

We first prove the following proposition which compares the eigenvalues corresponding to the pairs $(0,q+1)$ and $(1,q)$.

 \begin{prop}\label{p:0q1q}
 For any $h < -\frac{2}{\pi}$ and $q \in \N$, $q \geq 2$, we have
 \begin{equation}\label{eq:1q0q+1}
  \pi^{-2}(-\beta_0(h)^2 + \alpha_{q+1}(h)^2) > \pi^{-2}(-\beta_1(h)^2 + \alpha_q(h)^2).
  \end{equation}
 That is, the eigencurve corresponding to $(1,q)$ lies below that corresponding to $(0,q+1)$.
 \end{prop}

 In addition, by the proof of Lemma \ref{l:1}, we have that $$\lambda_{1,1,h}(S) < \lambda_{0,2,h}(S)\,,\, \forall h<-\frac{2}{\pi}\,.$$

 \begin{proof}~\\
By Lemma \ref{lem:monotonicity}(i), since $h \mapsto \alpha_p(h)^2 - \alpha_q(h)^2$ is increasing if $ p>q \geq 2$, we have
  \begin{equation*}
   \alpha_{q+1}(h)^2 - \alpha_q(h)^2 \geq \alpha_{q+1}(-\infty)^2 - \alpha_q(-\infty)^2 = \pi^2(2q-1) \geq 3\pi^2.
  \end{equation*}
 Again by Lemma \ref{lem:monotonicity}(iii), we have
  $$ \beta_0(h)^2 -\beta_1(h)^2 \leq \beta_0(-2/\pi)^2 -\beta_1(-2\pi)^2 \approx 5.7569 < 3\pi^2.$$
  Hence
  $$ \beta_0(h)^2 -\beta_1(h)^2 < \alpha_{q+1}(h)^2 - \alpha_q(h)^2$$
  as required.
 \end{proof}

\begin{prop}
For any  $N>0$, there exists $h_N < 0$ such that for $h< h_N$, the $4N $ first eigenvalues are given by the pairs  $(0,0)$, $(0,1)$, $(1,1)$, $(0,2)$, $(1,2)$, $\dots$, $(0,N)$, $(1,N)$. Moreover there exists $\hat h_N \leq h_N$ such that for $h<  \hat h_N$, these eigenvalues (except the first and fourth ones) have multiplicity $2$.
 \end{prop}

  \begin{proof}~\\
  We observe that $\lambda_{0,0,h}(S) < \lambda_{0,1,h}(S) < \lambda_{1,1,h}(S)$ and $\lambda_{0,2,h}(S) < \lambda_{1,2,h}(S)$ for all $h<-\frac{2}{\pi}$.
  By  the proof of Lemma \ref{l:1}, $\lambda_{1,1,h}(S) < \lambda_{0,2,h}(S)$ for all $h<-\frac{2}{\pi}$.

  \noindent
  So for any $h < -\frac{2}{\pi}$, we have that the first eight Robin eigenvalues of $S$ correspond to $(0,0)$, $(0,1)$, $(1,0)$, $(1,1)$, $(0,2)$, $(2,0)$, $(1,2)$, $(2,1)$ respectively.

  \noindent
  We also observe that $\lambda_{0,q,h}(S) < \lambda_{1,q,h}(S)$ as $ \beta_0(h) > \beta_1(h)$ for $h<-\frac{2}{\pi}$.

  \noindent
  The required ordering of the pairs now follows from Proposition \ref{p:0q1q}.

  \noindent
  We set $h_{N}$ to be the  value of $h < -\frac{2}{\pi}$ for which $\lambda_{1,N,h}(S) =0$.
  The existence of $h_{N}$ follows as $\lambda_{1,N,h}(S) \to -\infty$ as $h \to -\infty$.
  \end{proof}

Alternatively, we have the following proposition.
\begin{prop}\label{p:mult}
For any $h< - \frac 2\pi$, there exists $N(h)$ such that the negative spectrum  of the Robin Laplacian on $S$ consists either of eigenvalues given by the sequence $(0,0)$, $(0,1)$, $(1,1)$, $(0,2)$, $(1,2)$, $\dots$, $(0,N(h))$, $(1,N(h))$ or of  eigenvalues given by the sequence $(0,0)$, $(0,1)$, $(1,1)$, $(0,2)$, $(1,2)$, $\dots$, $(0,N(h))$. These eigenvalues have multiplicity $2$ (except  those corresponding to $(0,0)$, $(1,1)$).
\end{prop}

 We recall that there exists $-\frac{2}{\pi} < h_2^* < 0$, $h_2^* \approx -0.4382$, such that there is one negative Robin eigenvalue for $h_2^* \leq h < 0$ which corresponds to the pair $(0,0)$, and there are three negative Robin eigenvalues for $-\frac{2}{\pi} < h < h_2^*$ which correspond to the pairs $(0,0)$, $(0,1)$, $(1,0)$. All other Robin eigenvalues of the square are non-negative for $-\frac{2}{\pi} < h < 0$.

\subsection{The eigenvalues corresponding to pairs $(0,q)$, $q$ odd, or $(1,q)$.}\label{ss:0qodd}

In this subsection, we show that the negative eigenvalues of the Robin Laplacian on $S$ are not  Courant-sharp when they correspond to pairs of the form
\begin{itemize}
\item  $(0,q)$, $q$ odd,
\item  $(1,q)$, $q$ even,
\item  $(1,q)$, $q$ odd.
\end{itemize}

Let $N > 0$. By Proposition \ref{p:mult}, there exists $h_N<0$ such that for $h<h_N$, the $4N$ first eigenvalues are given by the pairs  $(0,0)$, $(0,1)$, $(1,1)$, $(0,2)$, $(1,2)$, $\dots$, $(0,N)$, $(1,N)$.

Any eigenfunction corresponding to $\lambda_{0,q,h}(S)$ with $q = 2j+1$, $j \in \N^*$, has the form (up to multiplication by a non-zero scalar):
$$ \Phi_{0,q,h,\theta}(x,y) :=  \cos \theta \cosh\left(\frac{\beta_0(h) x}{\pi}\right) \sin\left(\frac{\alpha_q(h)y}{\pi}\right) +  \sin \theta \cosh \left(\frac{\beta_0(h)y}{\pi}\right) \sin\left(\frac{\alpha_q(h) x}{\pi}\right)\, .$$
We see that $$ \Phi_{0,q,h,\theta}(-x,-y) =- \Phi_{0,q,h,\theta}(x,y)$$ so any such eigenfunction has an even number of nodal domains.

In this case, for $ \ell \in \N$, the eigenvalue $\lambda_{8\ell+9,h}(S)$ corresponds to the pair $(0,2(\ell +1) +1)$ (see Table \ref{tab:order}).
So the corresponding eigenfunction has an even number of nodal domains, but $8\ell +9$ is odd.

\begin{table}[h!]
	\centering
	\caption{The pairs $(p,q)$ corresponding to the eigenvalues $\lambda_{k,h}(S)$, $k \in \N$ with $h < 0$, $\vert h \vert$ large enough.
\label{tab:order}}
    \begin{tabular}{|c|c|c|c|c|c|c|c|c|c|c|c|c|c|c|c|c|c|c|c|c|c|c|c|c|c|}
    \hline
    $k$ & 1 & 2 & 3 & 4 & 5 & 6 & 7 & 8 & 9 & 10 & 11 & 12 & 13 & 14 & 15 & 16 & 17 & 18 & 19 \\
    \hline
    $p$ & 0 & 0 & 1 & 1 & 0 & 2 & 1 & 2 & 0 & 3 & 1 & 3 & 0 & 4 & 1 & 4 & 0 & 5 & 1 \\
    $q$ & 0 & 1 & 0 & 1 & 2 & 0 & 2 & 1 & 3 & 0 & 3 & 1 & 4 & 0 & 4 & 1 & 5 & 0 & 5 \\
    \hline
    \end{tabular}
\end{table}

Similarly, any eigenfunction corresponding to $\lambda_{1,q,h}(S)$ with $q = 2j$, $j \in \N^*$, has the form (up to multiplication by a non-zero scalar):
$$ \Phi_{1,q,h,\theta}(x,y) :=  \cos \theta \sinh\left(\frac{\beta_1(h) x}{\pi}\right) \cos\left(\frac{\alpha_q(h)y}{\pi}\right) +  \sin \theta \sinh \left(\frac{\beta_1(h)y}{\pi}\right) \cos\left(\frac{\alpha_q(h) x}{\pi}\right)\, .$$
We see that $$ \Phi_{1,q,h,\theta}(-x,-y) =- \Phi_{1,q,h,\theta}(x,y)$$ so any such eigenfunction has an even number of nodal domains.

In this case, for $ \ell \in \N$, the eigenvalue $\lambda_{8\ell+7,h}(S)$ corresponds to the pair $(1,2(\ell +1))$.
So the corresponding eigenfunction has an even number of nodal domains, but $8\ell +7$ is odd.

We therefore obtain the following proposition.
\newpage
\begin{prop}\label{p:0q}~
\begin{enumerate}
\item[(i)] For any $N>0$, there exists $h_N$ such that for $h\in (-\infty,h_N)$ the eigenvalue  of the Robin Laplacian on $S$ corresponding to the pair $(0,q)$ with  $q=3 + 2 \ell$ with $\ell=0,\dots, N$ has multiplicity $2$ and minimal labelling $(8 \ell +9)$ and is not Courant-sharp.
\item[(ii)] For any $N>0$, there exists $h_N$ such that for $h\in (-\infty,h_N)$ the eigenvalue of the Robin Laplacian on $S$ corresponding to the pair $(1,q)$ with  $ q= 2 + 2 \ell$ with $ \ell=0,\dots, N$ has multiplicity $2$ and minimal labelling $(8 \ell +7)$ and is not Courant-sharp.
\end{enumerate}
\end{prop}

Alternatively, one can write  the previous proposition in the following way.
\begin{prop}\label{p:0qbis}~
\begin{enumerate}
\item[(i)] For any $h< -\frac 2 \pi$, the negative eigenvalues  of the Robin Laplacian on $S$ corresponding to pairs of the form  $(0,q)$ with $q=3 + 2 \ell$ for some  $\ell\in \mathbb N$ have multiplicity $2$, minimal labelling $(8 \ell +9)$ and are not Courant-sharp.
\item[(ii)] For any $h< -\frac 2 \pi$, the negative eigenvalues  of the Robin Laplacian on $S$ corresponding to pairs of the form  $(1,q)$ with $ q= 2 + 2 \ell$ for some $\ell\in \mathbb N$ have  multiplicity $2$,  minimal labelling $(8 \ell +7)$ and are not Courant-sharp.
\end{enumerate}
\end{prop}

Finally for the eigenvalues corresponding to pairs of the form $(1,q)$, $q$ odd, we have  the following proposition.
\begin{prop}\label{p:1qbis}~
For any $h< -\frac 2 \pi$, the negative eigenvalues  of the Robin Laplacian on $S$ corresponding to pairs of the form  $(1,q)$ with  $q=3 + 2 \ell$ for some  $\ell\in \mathbb N$ have multiplicity $2$, minimal labelling $(8 \ell +11)$ and are not Courant-sharp.
\end{prop}
\begin{proof}
Any eigenfunction corresponding to such an eigenvalue has the form
$$ \Phi_{1,q,h,\theta}(x,y) :=  \cos \theta \sinh\left(\frac{\beta_1(h) x}{\pi}\right) \sin \left(\frac{\alpha_q(h)y}{\pi}\right) +  \sin \theta \sinh \left(\frac{\beta_1(h)y}{\pi}\right) \sin \left(\frac{\alpha_q(h) x}{\pi}\right)\, .$$
We observe that the $x$-axis and the $y$-axis belong to the nodal set of $\Phi_{1,q,h,\theta}(x,y)$.
Moreover, we see  that
$$
 \Phi_{0,q,h,\theta}(-x,y) = - \Phi_{0,q,h,\theta}(x,y) \,,\,  \Phi_{0,q,h,\theta}(x,-y) = - \Phi_{0,q,h,\theta}(x,y) \,.
$$
Hence the number of nodal domains is a multiple of $4$, but the labelling of the eigenvalue is congruent to $3$ modulo 4.
\end{proof}
 It remains to treat the eigenvalues of the Robin Laplacian on $S$ corresponding to pairs of the form $(0,q)$, $q$ even. This will be carried out in Section \ref{ss:0qeven}.

\section{The eigenvalues corresponding to pairs $(0,q)$, $q$ even.}\label{ss:0qeven}
\subsection{Preliminaries.}
We note that the negative eigenvalues corresponding to $(0, 2\ell +2)$, $\ell \in \N$, with labelling $8\ell +5$
are the eigenvalues corresponding to $(0,q)$, $q$ even, $q \geq 4$ with labelling $4q - 3$.
We wish to show that any corresponding eigenfunction has at most $4q-4$ nodal domains.
Then a consequence would be that any negative Robin eigenvalue corresponding to $(0,q)$, $q$ even, $q \geq 4$ is not Courant-sharp. Note that we are only considering the case of negative eigenvalues, hence we work under the assumption
\begin{equation}
\alpha_q(h) < \beta_0(h)\,.
\end{equation}
 Let $\tilde{h}_q$ be the value of $h$ such that $\beta_0(h) = \alpha_q(h)$.
The corresponding eigenfunctions associated  to the negative Robin eigenvalue that is given by the pair $(0,q)$, $\lambda_{0,q,h}(S)$, (up to scalar multiplication) are
\begin{equation}\label{defphi}
 \Phi_{0,q,h,\theta}(x,y) := \cos\theta \cosh\left(\frac{\beta_0(h) x}{\pi}\right) \cos\left(\frac{\alpha_q(h)y}{\pi}\right) + \sin\theta \cosh \left(\frac{\beta_0(h)y}{\pi}\right) \cos\left(\frac{\alpha_q(h) x}{\pi}\right)\, ,
 \end{equation}
 where it suffices to consider $\theta \in [0,\pi)$ as
 $$
  \Phi_{0,q,h,\pi+ \theta}(x,y) = - \Phi_{0,q,h,\theta}(x,y) \, .
  $$
 These are the only eigenfunctions corresponding to $\lambda_{0,q,h}(S)$ when it is negative because in this case, $\lambda_{0,q,h}(S)$ has multiplicity $2$ as $\beta_0(h) > \beta_1(h)$.

\subsection{Boundary zeros and interior critical zeros.}

\subsubsection{Boundary  zeros.}

From Proposition \ref{p:Sturmimp}, we deduce the following lemma.
\begin{lem}\label{lem:bdrypts}~
  \begin{itemize}
  \item[(i)] If $\Phi_{0,q,h,\theta}(\frac{\pi}{2},\frac{\pi}{2}) \neq 0$, then $\Phi_{0,q,h,\theta}(x,\frac{\pi}{2})$ has at most $q$ boundary zeros in $(-\frac{\pi}{2},\frac{\pi}{2})$.
  \item[(ii)] If $\Phi_{0,q,h,\theta}(\frac{\pi}{2},\frac{\pi}{2}) = 0$, then $\Phi_{0,q,h,\theta}(x,\frac{\pi}{2})$ has at most $(q-1)$ boundary zeros in $(-\frac{\pi}{2},\frac{\pi}{2})$.
  \end{itemize}
\end{lem}

\subsubsection{Interior critical zeros.}
We first consider the case where $\Phi_{0,q,h,\theta}$ does not have any interior critical zeros. We have the following proposition.

 \begin{prop}\label{lem:nointcritpts}
 If $\Phi_{0,q,h,\theta}$ is a Robin eigenfunction corresponding to a negative eigenvalue (i.e. $h < \tilde{h}_q$), which has no interior critical zeros and $q$ is even, $q \geq 4$, then it is not Courant-sharp.
 \end{prop}

 \begin{proof}
 By Theorem \ref{T-st2}, $\sum_{{\bf y}_i}\rho({\bf y}_i) \leq 4q$. We apply Euler's formula to obtain
 $$
 k \leq 1 + 2 q < 4q-3\,
 $$
 for $q$ even, $q \geq 4$.
 \end{proof}

 We now investigate the cases where $\Phi_{0,q,h,\theta}$ has interior critical zeros.

\begin{prop}\label{p:intcritexist}
For any eigenfunction corresponding to a  negative Robin eigenvalue which is given by the pair $(0,q)$ with $q$ even, $q \geq 4$, with $h<0$, the interior critical zeros $(x,y)$ are solutions of the equations:
  \begin{align}
    \beta_0 \sinh\left(\frac{\beta_0 x}{\pi}\right)\cos\left(\frac{\alpha_q x}{\pi}\right) + \alpha_q \cosh\left(\frac{\beta_0 x}{\pi}\right) \sin\left(\frac{\alpha_q x}{\pi}\right) & =0\,, \label{wx} \\
    \beta_0 \sinh\left(\frac{\beta_0 y}{\pi}\right)\cos\left(\frac{\alpha_q y}{\pi}\right) + \alpha_q \cosh\left(\frac{\beta_0 y}{\pi}\right) \sin\left(\frac{\alpha_q y}{\pi}\right) & =0\,. \label{wy}
  \end{align}
\end{prop}

\begin{proof}
  Any eigenfunction corresponding to  the negative Robin eigenvalue $\lambda_{0,q,h}(S)$ (up to scalar multiplication) has the form:
  $$ \Phi_{0,q,h,\theta}(x,y) :=  \cos \theta \cosh\left(\frac{\beta_0 x}{\pi}\right) \cos\left(\frac{\alpha_q y}{\pi}\right) +  \sin \theta \cosh \left(\frac{\beta_0 y}{\pi}\right) \cos\left(\frac{\alpha_q x}{\pi}\right)\, .$$

  At a zero critical point we have that $\Phi_{0,q,h,\theta}(x,y) = 0\,$, $\frac{\partial \Phi_{0,q,h,\theta}}{\partial x}(x,y) = 0$ and \break  $\frac{\partial \Phi_{0,q,h,\theta}}{\partial y}(x,y) = 0\,$. This reads
  \begin{equation}\label{eq:7}
  \Phi_{0,q,h,\theta}(x,y) =  \cos \theta \cosh\left(\frac{\beta_0 x}{\pi}\right) \cos\left(\frac{\alpha_q y}{\pi}\right) +  \sin \theta \cosh \left(\frac{\beta_0 y}{\pi}\right) \cos\left(\frac{\alpha_q x}{\pi}\right) = 0\, ,
  \end{equation}
   \begin{equation}\label{eq:7a2}
   \frac{\partial \Phi_{0,q,h,\theta}}{\partial x}(x,y) :=  \frac{\beta_0}{\pi}\cos \theta \sinh\left(\frac{\beta_0 x}{\pi}\right) \cos\left(\frac{\alpha_q y}{\pi}\right) - \frac{\alpha_q}{\pi} \sin \theta \cosh \left(\frac{\beta_0 y}{\pi}\right) \sin\left(\frac{\alpha_q x}{\pi}\right) = 0\, ,
   \end{equation}
   and
   \begin{equation}\label{eq:7b2}
   \frac{\partial \Phi_{0,q,h,\theta}}{\partial y}(x,y) :=  -\frac{\alpha_q}{\pi} \cos \theta \cosh\left(\frac{\beta_0 x}{\pi}\right) \sin\left(\frac{\alpha_q y}{\pi}\right) + \frac{\beta_0}{\pi} \sin \theta \sinh \left(\frac{\beta_0 y}{\pi}\right) \cos\left(\frac{\alpha_q x}{\pi}\right) =0\, .
   \end{equation}

   We see that $\cos \theta \neq 0$. If $\cos \theta = 0$ then by \eqref{eq:7}, $\cos\left(\frac{\alpha_q x}{\pi}\right) = 0$
   and by \eqref{eq:7a2}, $\sin\left(\frac{\alpha_q x}{\pi}\right)=0$ which is not possible.
   Similarly $\sin \theta \neq 0$ (via \eqref{eq:7} and \eqref{eq:7b2}).

   We also have that $\cos\left(\frac{\alpha_q x}{\pi}\right) \neq 0$.
   Indeed, if $\cos\left(\frac{\alpha_q x}{\pi}\right) = 0$ then $\sin\left(\frac{\alpha_q y}{\pi}\right)=0$ by \eqref{eq:7b2} but then $\Phi_{0,q,h,\theta}(x,y) \neq 0$.

   Similarly, if $x \neq 0$, then $\sin\left(\frac{\alpha_q x}{\pi}\right) \neq 0$.
   If $ \sin\left(\frac{\alpha_q x}{\pi}\right) = 0$, then by \eqref{eq:7a2}, $\cos\left(\frac{\alpha_q y}{\pi}\right) = 0$
   but then $\Phi_{0,q,h,\theta}(x,y) \neq 0$.

   If $x=0$, $y=0$, then $\frac{\partial \Phi_{0,q,h,\theta}}{\partial x}(0,0) = 0$ and $\frac{\partial \Phi_{0,q,h,\theta}}{\partial y}(0,0) = 0$ and we have that $\Phi_{0,q,h,\theta}(0,0) = 0$ if and only if $\theta=\frac{3\pi}{4}$.

   If $x=0$, and $y \neq 0$, then $\frac{\partial \Phi_{0,q,h,\theta}}{\partial x}(0,y) = 0$. By using
   $\frac{\partial \Phi_{0,q,h,\theta}}{\partial y}(0,y) = 0$, we obtain
   \begin{equation*}
    \sin \theta = \cos \theta \frac{\alpha_q \sin\left(\frac{\alpha_q y}{\pi}\right)}{\beta_0 \sinh\left(\frac{\beta_0 y}{\pi}\right)}.
   \end{equation*}
   By substituting the latter into \eqref{eq:7}, we obtain \eqref{wy}.

   For the remaining cases, we have that \eqref{eq:7} implies
   \begin{equation}\label{eq:7a}
  \tan \theta = \frac{-\cosh\left(\frac{\beta_0 x}{\pi}\right) \cos\left(\frac{\alpha_q y}{\pi}\right)}{\cosh \left(\frac{\beta_0 y}{\pi}\right) \cos\left(\frac{\alpha_q x}{\pi}\right)},
  \end{equation}
  \eqref{eq:7a2} implies
  \begin{equation}\label{eq:7b}
  \tan \theta = \frac{\beta_0 \sinh\left(\frac{\beta_0 x}{\pi}\right) \cos\left(\frac{\alpha_q y}{\pi}\right)}{\alpha_q \cosh \left(\frac{\beta_0 y}{\pi}\right) \sin\left(\frac{\alpha_q x}{\pi}\right)},
  \end{equation}
  and
  \eqref{eq:7b2} implies
  \begin{equation}\label{eq:7c}
  \tan \theta = \frac{\alpha_q \cosh\left(\frac{\beta_0 x}{\pi}\right) \sin\left(\frac{\alpha_q y}{\pi}\right)}{\beta_0 \sinh \left(\frac{\beta_0 y}{\pi}\right) \cos\left(\frac{\alpha_q x}{\pi}\right)}\,.
  \end{equation}
  At an interior critical zero, we have $ \Phi_{0,q,h,\theta}(x,y) = 0$ and $\nabla \Phi_{0,q,h,\theta}(x,y) = 0\,$.
  Therefore,  by equating \eqref{eq:7a} and \eqref{eq:7b}, resp. \eqref{eq:7a} and \eqref{eq:7c}, we obtain \eqref{wx}, resp. \eqref{wy}.
\end{proof}

We see that up to renormalising,
  \begin{equation}\label{defW}
  W(x) = \beta_0 \sinh\left(\frac{\beta_0 x}{\pi}\right)\cos\left(\frac{\alpha_q x}{\pi}\right) + \alpha_q \cosh\left(\frac{\beta_0 x}{\pi}\right) \sin\left(\frac{\alpha_q x}{\pi}\right)
  \end{equation}
is the Wronskian of $u_0$ and $u_q$.

\begin{rem}\label{remW}
 We observe that $x=0$ is a zero of $W(x)$. In addition, if $x=\gamma$ is a zero of $W(x)$, then $x=-\gamma$ is also a zero of $W(x)$ as $W(-x) = -W(x)$.
\end{rem}

\begin{lem}\label{lem:lbdzerosW}
   $W$ has at least $q-1$ zeros in $(-\frac{\pi}{2}, \frac{\pi}{2})$.
\end{lem}

\begin{proof}
  We consider $\frac{\alpha_q x_k}{\pi} = k\frac \pi 2$ for $k\in \mathbb Z$.
  We recall that $ (q-1)  \pi < \alpha_q (h) < q \pi$ for $h\in (-\infty,0)$. \\
  For $k>0$, we have $x_k =k  \frac \pi 2 \frac{\pi}{\alpha_q}$, hence
  $$k  \frac{\pi}{2q}  < x_k < k \frac {\pi}{2 (q-1)}.$$
  In light of Remark \ref{remW}, it suffices to consider the zeros in $(0, \frac{\pi}{2})$.
  So we consider $1 \leq k \leq q-1$.

  For $k = 4\ell +1$, $\ell \in \N$, we have $\frac{\alpha_q x_{4\ell+1}}{\pi} = (4\ell +1)\frac \pi 2$
  so $\cos\left(\frac{\alpha_q x}{\pi}\right) = 0$ and $\sin\left(\frac{\alpha_q x}{\pi}\right)>0$
  hence $W(x_{4\ell+1}) >0$.

  For $k = 4\ell +3$, $\ell \in \N$, we have $\frac{\alpha_q x_{4\ell+3}}{\pi} = (4\ell +3)\frac \pi 2$
  so $\cos\left(\frac{\alpha_q x}{\pi}\right) = 0$ and $\sin\left(\frac{\alpha_q x}{\pi}\right)<0$
  hence $W(x_{4\ell+1}) <0$.

  Thus, in $(x_1,x_{q-1})$, $W$ has at least $2 \left(\frac{q-2}{4}\right)$ zeros (as $4\ell +1 = q-1 \iff \ell = \frac{q-2}{4}$ and for each $\ell$ there are $2$ zeros, one in $(x_{4\ell+1},x_{4\ell+3})$ and one in $(x_{4\ell+3},x_{4\ell+5})$).

  Therefore, in $(-\frac{\pi}{2}, \frac{\pi}{2})$, $W$ has at least $2 \cdot 2 \left(\frac{q-2}{4}\right) + 1 = q-1$ zeros.
\end{proof}

\begin{rem}\label{rem:zerosdiag}
 By the proof of Lemma \ref{lem:lbdzerosW}, we see that for $1 \leq k \leq q-1$ and $k$ odd, $\cos\left(\frac{\alpha_q x_k}{\pi}\right) = 0$. This corresponds to $\frac{q-2}{2} + 1 = \frac{q}{2}$ zeros in $(0,\frac{\pi}{2})$.
\end{rem}

\begin{lem}\label{lem:ubdzerosW}
   $W$ has exactly $q-1$ zeros in $(-\frac \pi 2,\frac \pi 2)$\,.
\end{lem}

\begin{proof}
  Up to renormalising, we have that
  $$W'(x) = (\lambda_q - \lambda_0)u_0(x)u_q(x).$$
  As  $u_0$ does not change sign, and $u_q$ has $q$ zeros in $(-\frac{\pi}{2}, \frac{\pi}{2})$ by Sturm's theorem, $W'(x)$ has $q$ zeros in $(-\frac{\pi}{2}, \frac{\pi}{2})$.
  Hence $W(x)$ has at most $(q-1)$ zeros in $(-\frac{\pi}{2}, \frac{\pi}{2})$.
  Since $W(x)$ has at least $(q-1)$ zeros in $(-\frac{\pi}{2}, \frac{\pi}{2})$ by Lemma \ref{lem:lbdzerosW}, we obtain the desired conclusion.
\end{proof}

 By the proof of Lemma \ref{lem:lbdzerosW}, we obtain the localisation of the zeros of $W$.

\begin{lem}\label{rem:localisation}
Let $-\gamma_{(q-2)/2},\dots ,-\gamma_1, 0, \gamma_1, \dots , \gamma_{(q-2)/2}\,$ denote the zeros of $W$.
For $\ell =1, \dots, (q-2)/2\,,$
\begin{enumerate}
\item[(i)] \begin{equation}\label{eq:enc}
\alpha_q \gamma_1/\pi \in \left(\frac \pi 2,\pi\right)\,, \, \cdots\,,\, \alpha_q \gamma_\ell /\pi \in \left(\frac {(2\ell -1)\pi}{2} \,,\,  \ell  \pi \right), \dots
\end{equation}
\item[(ii)] \begin{equation}\label{eq:enc2}
\lim_{h\rightarrow -\infty} \gamma_\ell (h) =  (2\ell -1) \frac { \pi}{ 2 (q-1)}\,.
\end{equation}
\item[(iii)] As $h\rightarrow -\infty$,
\begin{equation}\label{asgammaj}
\frac{\alpha_q(h) \gamma_\ell(h)}{\pi}  -\frac{ (2\ell -1) \pi}{ 2} \sim (q-1) (-h) ^{-1}\,.
\end{equation}
\end{enumerate}
\end{lem}
\begin{proof}~
 \begin{enumerate}
    \item[(i)] By the proof of Lemma \ref{lem:lbdzerosW}, we have $\alpha_q \gamma_1/\pi \in (\frac \pi 2,\frac{3\pi}{2})$, $\alpha_q \gamma_2/\pi \in (\frac {3\pi} 2,\frac{5\pi}{2}), \dots $.\newline
        Observing that  $\cos (\alpha_q \gamma_\ell /\pi) \, \sin (\alpha_q \gamma_\ell /\pi) <0$ by \eqref{wx}, we obtain \eqref{eq:enc}.
    \item[(ii)] In order to ensure that $W(\gamma_\ell(h)) = 0$ as $h \to -\infty$, we recall that $\beta_0(h) \to +\infty$ as $h \to -\infty$, and we must have
        \begin{equation*}
          \lim_{h \to -\infty} \cos \left(\frac{\alpha_q \gamma_\ell(h)}{\pi}\right) = 0.
        \end{equation*}
        \end{enumerate}
        Now $\frac{\alpha_q}{\pi} \to q-1$ as $h \to -\infty$ and $\cos ((q-1)x) = 0$ if $(q-1)x = \frac{(2\ell -1)\pi}{2}$. So we obtain \eqref{eq:enc2}.
    \item[(iii)] Starting from
     \begin{equation}\label{neweq}
     \beta_0 \sinh\left(\frac{\beta_0 \gamma_\ell }{\pi}\right)\cos\left(\frac{\alpha_q \gamma_\ell }{\pi}\right) + \alpha_q \cosh\left(\frac{\beta_0 \gamma_\ell}{\pi}\right) \sin\left(\frac{\alpha_q \gamma_\ell}{\pi}\right)=0\, ,
     \end{equation}
     by \eqref{eq:beta0} and \eqref{encad}, we obtain that as $h \to -\infty$,
     \begin{equation}\label{eq:8.20}
    \cos \left(\frac{\alpha_q \gamma_\ell(h)}{\pi}\right)
    = -\frac{\alpha_q \tanh\left(\frac{\beta_0 \gamma_\ell(h)}{\pi}\right) \sin \left(\frac{\alpha_q \gamma_\ell(h)}{\pi}\right)}{\beta_0}
    \sim \frac{\pi(q-1) (-1)^{2\ell -1}}{(-\pi h)}.
     \end{equation}
    Hence, as $h \to -\infty$, we have \eqref{asgammaj}. \\~\\
    We can improve \eqref{asgammaj} as follows. If we start from the left-hand side of \eqref{eq:8.20},
     then, with $\hat \gamma_{\ell,q} = \frac{\alpha_q \gamma_\ell(h)}{\pi}$ and $\epsilon_q = \frac{\alpha_q }{\beta_0}$, this reads
       \begin{equation*}
       \cot  \hat \gamma_{\ell,q}
       = - \epsilon_q \,  \tanh\left( \frac{\hat  \gamma_{\ell,q}}{\epsilon_q}\right)\,,
       \end{equation*}
       and
       \begin{equation*}
       \hat \gamma_{\ell,q} = \frac{ (2\ell -1) \pi}{ 2} -
       \arctan \left ( \epsilon_q \tanh \frac{\beta_0\gamma_\ell }{\pi}\right)\,.
       \end{equation*}
 Note moreover that we have
 $$
 \arctan \left ( \epsilon_q \tanh \frac{\beta_0\gamma_\ell }{\pi}\right) - \arctan \epsilon_q \sim  -   2 \epsilon_q \exp -\frac{2 \beta_0\gamma_\ell }{\pi} \,.
 $$
 Hence
  \begin{equation}\label{refineda}
   \hat \gamma_{\ell,q} = \frac{ (2\ell -1) \pi}{ 2} - \arctan \epsilon_q  + 2 \epsilon_q (1+o(1))  \exp -\frac{2 \beta_0\gamma_\ell }{\pi} \,.
 \end{equation}
 We observe that up to an exponentially small term $\frac{\alpha_q \gamma_\ell(h)}{\pi} $ is an affine function with respect to $\ell$.
      \end{proof}

  Analogous arguments to those in the proof of Lemma \ref{lem:lbdzerosW} hold for \eqref{wy}. Hence we have proved the following proposition.

\begin{prop}\label{p:intcritpts}
Any eigenfunction corresponding to a negative Robin eigenvalue which is given by the pair $(0,q)$ with $q$ even, $q \geq 4$, $ h<0$, has at most $(q-1)^2$ interior critical zeros.
\end{prop}

\begin{rem}\label{rem4sym}
  We observe that if $ (x,y)$ is an interior critical zero of $\Phi_\theta$ then $(-x,y)$, $(x,-y)$, $(-x, -y)$ are also interior critical zeros. In particular, if $ xy\neq0$, we get four distinct interior critical zeros.
\end{rem}

\begin{rem}
The estimate on the number of interior critical zeros given in Proposition~\ref{p:intcritpts} is too rough.
By using a $\theta$-independent condition we obtain a set of possible interior critical zeros corresponding to all values of $\theta$. However, we are actually interested in the supremum over $\theta$ of the number of interior critical zeros of $\Phi_\theta$.\\
In addition, we must count ``with the degree of singularity''.
This is the same and equal to $2$ if we show that the interior critical zeros are non-degenerate.\\
We note that the improvement of Sturm's theorem, Proposition \ref{p:Sturmimp}, gives an upper bound for the quantity $\sum_{{\bf y}_i}\rho({\bf y}_i)$ as it counts the zeros with their multiplicities.
\end{rem}

\begin{lem}\label{lemnondeg}
If $\cos \theta\neq 0$, all the interior critical zeros of $\Phi_{\theta}$ are non-degenerate.
\end{lem}

\begin{rem}\label{rem:pion2}
The case where $\cos \theta = 0$ (i.e. $\theta=\frac{\pi}{2}$) can be treated directly and $\Phi_{\theta}$ has
$q+1$ nodal domains (By Sturm's Theorem $u_{q+1,h}$ has $q$ zeros in $(-\frac{\pi}{2},\frac{\pi}{2})$).
 As $q+1 < 4q-3$ for $q > \frac{4}{3}$, $\Phi_{0,q,h,\frac{\pi}{2}}(x,y)$ is not Courant-sharp for $h < \tilde{h}_q$.
\end{rem}

\begin{proof}
If $\Phi_{0,q,h,\theta}(x,y) = 0$, then
\begin{align*}
\frac{\partial^2 \Phi_{0,q,h,\theta}}{\partial x^2}(x,y) &=  \frac{\beta_0^2}{\pi^2}\cos \theta \cosh\left(\frac{\beta_0 x}{\pi}\right) \cos\left(\frac{\alpha_q y}{\pi}\right) - \frac{\alpha_q^2}{\pi^2} \sin \theta \cosh \left(\frac{\beta_0 y}{\pi}\right) \cos\left(\frac{\alpha_q x}{\pi}\right) \\
&= \left( \frac{\beta_0^2}{\pi^2} + \frac{\alpha_q^2}{\pi^2} \right) \cos \theta \cosh\left(\frac{\beta_0 x}{\pi}\right) \cos\left(\frac{\alpha_q y}{\pi}\right),
\end{align*}
and
\begin{align*}
\frac{\partial^2 \Phi_{0,q,h,\theta}}{\partial y^2}(x,y) &=  -\frac{\alpha_q^2}{\pi^2}\cos \theta \cosh\left(\frac{\beta_0 x}{\pi}\right) \cos\left(\frac{\alpha_q y}{\pi}\right) + \frac{\beta_0^2}{\pi^2} \sin \theta \cosh \left(\frac{\beta_0 y}{\pi}\right) \cos\left(\frac{\alpha_q x}{\pi}\right) \\
&= -\left( \frac{\beta_0^2}{\pi^2} + \frac{\alpha_q^2}{\pi^2} \right) \cos \theta \cosh\left(\frac{\beta_0 x}{\pi}\right) \cos\left(\frac{\alpha_q y}{\pi}\right).
\end{align*}

So the Hessian has the form $\left(\begin{array}{cc} m_{11}&m_{12}\\m_{12}& -m_{11}\end{array}\right)$.\\
We note that  the trace of the Hessian is zero since $-\frac{\partial^2}{\partial x^2}\Phi_{0,q,h,\theta}(x,y) - \frac{\partial^2}{\partial y^2}\Phi_{0,q,h,\theta}(x,y) = \lambda_{0,q,h,\theta} \Phi_{0,q,h,\theta}(x,y) = 0 $ at a zero  point.
The determinant of the Hessian is $-m_{11}^2 -m_{12}^2$. Hence the fact that the determinant is non-zero can be deduced from the condition $m_{11}\neq 0$. Here
$$
m_{11} =  \left( \frac{\beta_0^2}{\pi^2} + \frac{\alpha_q^2}{\pi^2} \right) \cos \theta \cosh\left(\frac{\beta_0 x}{\pi}\right) \cos\left(\frac{\alpha_q y}{\pi}\right)\,.
$$
We have $ \cos\left(\frac{\alpha_q y}{\pi}\right) \neq 0$ if $y$ is a zero of $W$ as a direct consequence of \eqref{defW}.
\end{proof}

\begin{prop}\label{prop:tannot1}
 For all $h < 0$, for all even $q\geq 4$, there exists $\epsilon >0$, such that, for all $\theta \in [0,\frac \pi 4 +\epsilon)$ there are no interior critical zeros of $\Phi_{0,q,h,\theta}$ on the $x$-axis.
 \newline
 Analogously, for all $h < 0$, for all even $q\geq 4$, there exists $\epsilon >0$, such that, for all $\theta \in (\frac \pi 4 -\epsilon,\frac \pi 2]$ there are no interior critical zeros of $\Phi_{0,q,h,\theta}$ on the $y$-axis.
\end{prop}

\begin{proof}
  We have that
  $$ \Phi_{0,q,h,\theta}(x,0) := \sin \theta \, \cos\left(\frac{\alpha_q(h)x}{\pi}\right) + \cos \theta \, \cosh \left(\frac{\beta_0(h)x}{\pi}\right).$$
  For $\theta \in [ 0,  \frac{\pi}{4}] $, $\Phi_{0,q,h,\theta}(x,0) > 0$. To extend this inequality  to a larger interval  of $\theta$, we have only to prove it for each possible interior critical zero on the $x$-axis.
  If $x\neq 0$, it is enough to observe that $\cosh ( \beta_0(h) x/\pi) >0$.
  For $x=0$, we just observe that $\cos \theta +\sin \theta >0$.
  It is then clear by continuity that there exists some $\epsilon >0$ such that $\Phi_{0,q,h,\theta}(x,0) > 0$ for these interior critical zeros on the interval.

  On the $y$-axis, we have that
    $$ \Phi_{0,q,h,\theta}(0,y) :=  \cos \theta \cos\left(\frac{\alpha_q(h) y}{\pi}\right) +  \sin \theta \cosh \left(\frac{\beta_0(h)  y}{\pi}\right) \, .$$
    It is enough to observe that
     $ \Phi_{0,q,h,\theta}(0,y)= \Phi_{0,q,h,\frac \pi 2-\theta}(y,0) $ and to apply the previous argument.
\end{proof}

\begin{prop}\label{prop:tan-1}
  A point of the form $(x,x)$ is an interior critical zero of $ \Phi_{0,q,h,\theta}(x,y)$ if and only if $\tan \theta = -1$.
\end{prop}

\begin{proof}
  Suppose $(x,x)$ is an interior critical zero of $ \Phi_{0,q,h,\theta}(x,y)$. Then
  $$ 0 = \Phi_{0,q,h,\theta}(x,x) := (\cos\theta + \sin\theta) \cosh \left(\frac{\beta_0(h)x}{\pi}\right) \cos\left(\frac{\alpha_q(h) x}{\pi}\right) \, .$$
  Now $\cosh \left(\frac{\beta_0(h)x}{\pi}\right) > 0$ and $\cos\left(\frac{\alpha_q(h) x}{\pi}\right) \neq 0$ by \eqref{wx} as $(x,x)$ is an interior critical zero. So $\cos \theta + \sin \theta = 0$, i.e. $\tan \theta = -1$.

  On the other hand, if there is no point of the form $(x,x)$ in the nodal set of $ \Phi_{0,q,h,\theta}(x,y)$, then
  $$ \Phi_{0,q,h,\theta}(x,x) := (\cos\theta + \sin\theta) \cosh \left(\frac{\beta_0(h)x}{\pi}\right) \cos\left(\frac{\alpha_q(h) x}{\pi}\right) \neq 0 \, ,$$
  which implies that $\cos \theta + \sin \theta \neq 0$, i.e. $\tan \theta \neq -1$.
\end{proof}

  The analogue of Proposition \ref{prop:tan-1} also holds for points of the form $(x,-x)$.

\begin{lem}\label{lem:3pion4}
For $\theta = \frac{3\pi}{4}$, $\Phi_{0,q,h,\frac{3\pi}{4}}$ has no interior critical zeros of the form
$(\gamma_i,0)$, $(0,\gamma_i)$, or $(\gamma_i,\gamma_j)$ where $i$ and $j$ are of different parity.
\end{lem}

\begin{proof}
  First consider critical zeros of the form $(\gamma_i,0)$ and let $\theta_i=\theta(\gamma_i,0)$.
  Then $(\gamma_i,0)$ corresponds to
 \begin{equation}\label{eq:8.18a}
 \tan \theta_i= -\cosh (\beta_0 \gamma_i/\pi)/ \cos ( \alpha_q \gamma_i/\pi)
 \end{equation}
 by \eqref{eq:7a}.

 If $\tan \theta_i = -1$, then $\cosh (\beta_0 \gamma_i/\pi) = \cos ( \alpha_q \gamma_i/\pi)$.
 But $\cos ( \alpha_q \gamma_i/\pi) \leq 1$, $\cosh (\beta_0 \gamma_i/\pi) \geq 1$, and $\cosh (\beta_0 \gamma_i/\pi) = 1$ if and only if $\gamma_i = 0$ which is not the case. So $\tan \theta_i \neq -1$ and $\theta_i \neq \frac{3\pi}{4}$.

 Next consider critical zeros of the form $(0,\gamma_i)$ and let $\hat{\theta}_i=\theta(0,\gamma_i)$.
 Then $\tan \theta_i \tan \hat{\theta}_i = 1$, and $\tan \theta_i \neq -1$ so $\tan \hat{\theta}_i \neq -1$ and $\hat{\theta}_i \neq \frac{3\pi}{4}$.

 Now consider critical zeros of the form $(\gamma_i,\gamma_j)$ where $i$ and $j$ have different parity.
 We have
 \begin{equation*}
  \tan \theta(\gamma_i,\gamma_j) = -\tan \theta(\gamma_i,0) \tan \theta(0,\gamma_j).
 \end{equation*}
 In addition, by Proposition \ref{rem:localisation},
 \begin{equation*}
   \alpha_q \gamma_j /\pi \in \left(\frac {(2j -1)\pi}{2} \,,\,  j  \pi \right)
 \end{equation*}
 so $\cos(\alpha_q \gamma_i/ \pi) \cos(\alpha_q \gamma_j/ \pi) < 0$.
 Hence $\tan \theta(\gamma_i,\gamma_j) > 0$ and therefore $\theta(\gamma_i,\gamma_j) \neq \frac{3\pi}{4}$.
\end{proof}

\begin{rem}By Proposition \ref{p:intcritpts} and Lemma \ref{lem:3pion4} with $q=6$, we deduce that $\Phi_{0,6,h,\frac{3\pi}{4}}(x,y)$ has at most $25 - 4 - 4 - 8 =9$ interior critical zeros.
By Lemma \ref{lem:bdrypts}, there are at most $20$ boundary zeros.
So by Euler's formula, we have
\begin{equation*}
  k \leq 1 + 9 + 10 = 20 < 21.
\end{equation*}
 Hence $\Phi_{0,6,h,\frac{3\pi}{4}}(x,y)$ is not Courant-sharp.

However, by similar considerations for $q=8$, $\Phi_{0,8,h,\frac{3\pi}{4}}(x,y)$ has at most
$49 - 6 - 6 - 8 - 8 = 21$ interior critical zeros and at most $28$ boundary zeros. Hence, we get
\begin{equation*}
  k \leq 1 + 21 + 14 = 36 \,,
\end{equation*}
to compare with $29$.

So Lemma \ref{lem:3pion4} is not strong enough to show that $\Phi_{0,q,h,\frac{3\pi}{4}}(x,y)$ is not Courant-sharp for $q \geq 8$. We need to show that if $i$ and $j$ have the same parity and $i \neq j$, then $\tan \theta(\gamma_i,\gamma_j) \neq -1$. We will show this in the case where $h<0$ with $\vert h \vert$ sufficiently large by considering the $(-h)$ large asymptotics.
\end{rem}

We now analyse the asymptotic behaviour of the $\theta(\gamma_i,\gamma_j) $ in order to compare them asymptotically as $h \to -\infty$.

 \begin{prop}\label{prop:asmtantheta}
 For $j=1,\dots, \frac q2-1$, we have that
 \begin{equation}\label{eq:8.18f}
   \tan \theta_j \sim  (-1)^{ j+1} \frac{\beta_0}{2\alpha_q \cos \arctan \epsilon_q}  e^{\beta_0 \gamma_j/\pi}  \,.
   \end{equation}
 \end{prop}
\begin{proof}
   By \eqref{eq:7b}, we have that
  \begin{equation}\label{eq:8.18bisa}
 \tan \theta_j = \frac{\beta_0}{\alpha_q} \sinh (\beta_0 \gamma_j/\pi)  / \sin ( \alpha_q \gamma_j/\pi)\, ,
 \end{equation}
 and we recall from \eqref{refineda} that
     \begin{equation} \label{refinedaa}
     \begin{array}{ll}
     \alpha_q \gamma_j/\pi&= \frac{ (2j -1) \pi}{ 2} - \arctan \left( \epsilon_q \tanh \left(\frac{\beta_0\gamma_j}{\pi}\right) \right) \\
      & = \frac{ (2j -1) \pi}{ 2} - \arctan \epsilon_q  + 2 \epsilon_q (1+o(1))  \exp \left(-\frac{2 \beta_0\gamma_j }{\pi}\right)
     \,,
     \end{array}
 \end{equation}
  with $\epsilon_q= \alpha_q/\beta_0=o(1)$.\\
 This implies
 \begin{equation*}
\sin (  \alpha_q \gamma_j/\pi)= (-1)^{j+1} \cos \left( \arctan \left( \epsilon_q \tanh \left( \frac{\beta_0\gamma_j}{\pi}\right) \right)\right)\,,
\end{equation*}
and
\begin{equation*}
\sin (  \alpha_q \gamma_j/\pi)= (-1)^{j+1} \cos \arctan \epsilon_q + \mathcal O (\epsilon_q^2) \exp \left(-\frac{ 2 \beta_0\gamma_j }{\pi}\right) \,.
\end{equation*}
From this we obtain that as $h \to -\infty$,
 \begin{equation}\label{eq:8.18tera}
 \tan \theta_j =  (-1)^{ j+1} \frac{\beta_0}{2\alpha_q \cos\arctan  \epsilon_q}   e^{\beta_0\gamma_j/\pi}(1- e^{- 2 \beta_0\gamma_j/\pi}) \,\left(1+ \mathcal O (\epsilon_q^2) \exp \left(-\frac{2 \beta_0\gamma_j }{\pi}\right) \right) \,.
 \end{equation}
 \end{proof}

 We wish to show that there exists $\hat{h}_q <0$ such that for $h < \hat{h}_q$, all $\theta_{i0}=\theta(\gamma_i,0)$, $\theta_{0i}=\theta(0,\gamma_i)$, $\theta_{ij}=\theta(\gamma_i,\gamma_j)$
 are distinct and different from $\frac{3\pi}{4}$.

 Due to the symmetry properties of the interior critical zeros, Remark \ref{remW} and Remark \ref{rem4sym}, it is sufficient to consider the cases where $\gamma_i > 0$.

 From \eqref{eq:8.18tera}, we see that for $h<0$, $\vert h\vert$ sufficiently large, $\tan \theta_{j0} \neq \pm 1$ so $\theta_{j0} \neq \theta_{0j}$. Indeed, if $\theta_{j0} = \theta_{0j}$, then $(\tan \theta_{j0})^2 = 1$ which is not possible.

 We are now interested in the quotient  of $\tan \theta_j$ and $\tan \theta_k$. For $j\neq k$, we get
  \begin{align}\label{eq:8.18terb}
 \tan \theta_j /\tan \theta_k &=  (-1)^{j-k}  e^{\beta_0(\gamma_j-\gamma_k)/\pi}\frac{(1- e^{-2\beta_0\gamma_j  /\pi}) }{(1-e^{-2\beta_0\gamma_k  /\pi})} \notag\\
  & \qquad \qquad \times \left(1+ \mathcal O (\epsilon_q^2) \exp \left(-\frac{2 \beta_0\gamma_j }{\pi}\right) +  \mathcal O (\epsilon_q^2) \exp \left(-\frac{ 2 \beta_0\gamma_k }{\pi}\right)\right) \,.
 \end{align}
 We take the logarithm and obtain
 \begin{align*}
\log \left ( (-1)^{j-k} \tan \theta_j /\tan \theta_k\right) &= \beta_0(\gamma_j-\gamma_k)/\pi  + e^{-2\beta_0  \gamma_k/\pi}- e^{-2\beta_0  \gamma_j /\pi} \\
 & \ \ \ \ \ \ + o (1) \exp \left(-\frac{2 \beta_0\gamma_j }{\pi}\right) + o (1) \exp \left(-\frac{2 \beta_0\gamma_k }{\pi}\right)  \,.
 \end{align*}

 On the other hand, dividing \eqref{refinedaa} by $\epsilon_q = \frac{\alpha_q}{\beta_0}$ leads to
  \begin{equation*}
  \beta_0 \gamma_j/\pi  = \frac{ (2j -1) \beta_0 \pi}{ 2\alpha_q } - \frac{\beta_0 \arctan \epsilon_q }{\alpha_q}  + 2  (1+o(1))
  \exp \left(-\frac{2 \beta_0\gamma_j }{\pi}\right)  \,.
\end{equation*}
This finally leads to
 \begin{equation}\label{eq:8.18tere}
 \begin{array}{l}
\log \left ( (-1)^{j-k} \tan \theta_j /\tan \theta_k\right) \\
\qquad =   \frac{\beta_0}{\alpha_q} (j-k) { \pi}  - e^{-2 \beta_0 \gamma_k/\pi} + e^{-2\beta_0 \gamma_j/\pi} + o (1) \exp\left(-\frac{2 \beta_0\gamma_j }{\pi}\right) +  o (1) \exp \left(-\frac{2 \beta_0\gamma_k }{\pi}\right)
 \,.
 \end{array}
 \end{equation}

 By \eqref{eq:8.18tere},  it follows immediately that there exists $\hat{h}_q <0$ such that for $h < \hat{h}_q$, if $j < k$, then $\theta_j=\theta_{j0} \neq \theta_{k0}=\theta_k$.
 This also gives that $\theta_{jk} \neq \frac{3\pi}{4}$.
  By using \eqref{eq:8.18f}, it also follows that there exists $\hat{h}_q <0$ such that for $h < \hat{h}_q$, $\theta_{j0} \neq \theta_{0k}$, $j<k$.
 In addition, we can show that there exists $\hat{h}_q <0$ such that for $h < \hat{h}_q$, $\theta_{i0} \neq \theta_{jk}$, $j \leq k$.

 The most difficult case is to show that $\theta_{jk} \neq \theta_{j'k'}$. We wish to compare the quantities
 \begin{equation*}
 \sigma_{jk} (h):=\log \left ( (-1)^{j-k} \tan \theta_j /\tan \theta_k\right) \mbox{ and }
 \sigma_{j'k'}(h):=\log \left ( (-1)^{j'-k'} \tan \theta_{j'} /\tan \theta_{k'}\right)\,.
 \end{equation*}
If $j-k \neq j'-k'$ it follows from \eqref{eq:8.18tere} that $\theta_{jk} \neq \theta_{j'k'}$ for $-h$ large enough.

 Now suppose that $j-k=j'-k'$ and $j <k$ and $j  < j' <k'$.
 We obtain from above that
 $$
 \sigma_{jk}(h) -\sigma_{j'k'}(h) = - e^{-2\beta_0  \gamma_k/\pi} + e^{-2\beta_0 \gamma_j/\pi}  +
 e^{-2 \beta_0 \gamma_k'/\pi}  -
 e^{-2\beta_0  \gamma_j' /\pi}
  + o (1) \exp \left(- 2  \beta_0 \gamma_j/\pi \right),
 $$
 where we use Lemma \ref{rem:localisation} which gives that $\gamma_j < \gamma_k, \gamma_j', \gamma_k'$.
 For $-h$ large, this quantity is equivalent to $ e^{-2 \beta_0  \gamma_j/\pi} $ and consequently has positive sign.
 This leads to the following lemma.
 \begin{lem}
 If $j <k$, $j' <k'$ and $j-k=j'-k'$, then there exists $h^*_q< 0$ such that $\sigma_{jk}(h) -\sigma_{j'k'}(h) > 0$ for $h\leq  h^*_q$.\\
 In particular,
 $$\tan \theta_j /\tan \theta_k \neq  \tan \theta_{j'} /\tan \theta_{k'}\,,\,\forall h\in (-\infty,h^*_q ].$$
 \end{lem}

The above discussion leads to the following proposition.

 \begin{prop}\label{CSq}
 For each even $q\geq 4$, there exists $h_q< 0$ such that for $h\leq  h_q$, no eigenfunction associated with $\lambda_{0,q,h}(S)$ is Courant-sharp.
 \end{prop}

  \begin{proof}[Proof of Proposition \ref{CSq}]
   If $\Phi_{0,q,h,\theta}(x,y)$ has no interior critical zeros, then Proposition \ref{lem:nointcritpts} applies to give that $\Phi_{0,q,h,\theta}(x,y)$ is not Courant-sharp.

   In the case where $\Phi_{0,q,h,\theta}(x,y)$ has interior critical zeros, the above arguments give the existence of $\bar{h}_q$ such   that $\theta(\gamma_i,0)$, $\theta(0,\gamma_i)$, $\theta(\gamma_i,\gamma_j)$ are distinct for $h < \bar{h}_q$.
   We wish to count the number of potential interior critical zeros corresponding to each such $\theta$ in order to apply Euler's formula. For $\theta \neq \frac{\pi}{2}$, we showed that the interior critical zeros are non-degenerate (Lemma \ref{lemnondeg}). The case of $\theta = \frac{\pi}{2}$ was treated in Remark \ref{rem:pion2}.

   For $i \in \{-\frac{(q-2)}{2}, \dots, -1, 1, \dots, \frac{(q-2)}{2}\}$, if $(\gamma_i,0)$ is an interior critical zero then $(-\gamma_i,0)$ is also an interior critical zero for the same value of $\theta = \theta_{i0}$ by \eqref{eq:7a}. Since the critical $\theta$'s are distinct, there are only $2$ interior critical zeros of $\Phi_{0,q,h,\theta_{i0}}(x,y)$. By Lemma \ref{lem:bdrypts}, there are at most $4q$ boundary zeros.
   So, by Euler's formula, we have
   \begin{equation*}
     k \leq 1 + 2 + 2q < 4q-3
   \end{equation*}
   for $q > 3$. Hence $\Phi_{0,q,h,\theta_{i0}}(x,y)$ is not Courant-sharp.

   Analogous arguments for the case with interior critical zeros of the form $\{(0,\gamma_i), (0,-\gamma_i)\}$ give that $\Phi_{0,q,h,\theta_{0i}}(x,y)$ is not Courant-sharp where $\theta_{0i} = \theta(0,\gamma_i)$.

   For $i, j \in \{-\frac{(q-2)}{2}, \dots, -1, 1, \dots, \frac{(q-2)}{2}\}$ such that $i \neq j$, if $(\gamma_i,\gamma_j)$ is an interior critical zero of $\Phi_{0,q,h,\theta_{ij}}(x,y)$, then $(-\gamma_i,\gamma_j)$, $(\gamma_i,-\gamma_j)$, $(-\gamma_i,-\gamma_j)$ are also interior critical zeros (Remark \ref{rem4sym}) where $\theta_{ij} = \theta(\gamma_i,\gamma_j)$. Since the critical $\theta$'s are distinct, $\Phi_{0,q,h,\theta_{ij}}(x,y)$ has $4$ interior critical zeros. By Lemma \ref{lem:bdrypts}, as $i \neq j$, there are at most $4q$ boundary zeros. Hence by Euler's formula,
   \begin{equation*}
     k \leq 1 + 4 + 2q = 2q + 5 < 4q - 3
   \end{equation*}
   for $q > 4$. Hence $\Phi_{0,q,h,\theta_{ij}}(x,y)$ is not Courant-sharp.

   It remains to treat the case $\theta = \frac{3\pi}{4}$.
   By Proposition \ref{prop:tan-1} and Lemma \ref{lem:ubdzerosW}, $\Phi_{0,q,h,\frac{3\pi}{4}}$ has at least $2(q-2)+1 = 2q-3$ interior critical zeros.
   By Lemma \ref{lem:3pion4}, $\Phi_{0,q,h,\frac{3\pi}{4}}$ has no interior critical zeros of the form $(\gamma_i,0)$, $(0,\gamma_i)$, or $(\gamma_i,\gamma_j)$ where $i$ and $j$ are of different parity.
   In the preceding discussion, we showed that there exists $\hat{h}_q <0$ such that for $h < \hat{h}_q$, $\theta_{ij} \neq \frac{3\pi}{4}$.
   Therefore, for $h < \bar{h}_q$, $\Phi_{0,q,h,\frac{3\pi}{4}}$ has exactly $2q -3$ interior critical zeros.
   By Lemma \ref{lem:bdrypts}, $\Phi_{0,q,h,\frac{3\pi}{4}}$ has at most $4(q-1)$ boundary zeros.
   By Euler's formula, we have
   \begin{equation*}
     k \leq 1 + 2q -3 + 2(q-1) = 4q - 4 < 4q -3.
   \end{equation*}
   So $\Phi_{0,q,h,\frac{3\pi}{4}}$ is not Courant-sharp.
 \end{proof}
\begin{rem}
The proof of Proposition \ref{CSq} holds more generally for any $h<0$ such that the $\theta(\gamma_i,\gamma_j)$ are distinct.
\end{rem}

\subsection{Detailed analysis for $q=4$.}

In this subsection, we consider the Robin eigenvalue of $S$ corresponding to the pair $(0,4)$ for $h < 0$.
In order to see what is at stake, we first consider the nodal partitions for the case where $h=-4$.
We then give a general argument which shows that the negative Robin eigenvalue given by the pair $(0,4)$ is not Courant-sharp.

\subsubsection{Particular case $h=-4$.}
We note that with $h=-4$, the first $16$ Robin eigenvalues are given by the pairs
$(0,0)$, $(0,1)$, $(1,1)$, $(0,2)$, $(1,2)$, $\dots$, $(0,4)$, $(1,4)$.
Numerically, with $h=-4$, we compute that $ \gamma \approx 0.6625$ is a solution of $W(x)=0$.
Since $W(-x) = -W(x)$, $x=- \gamma \approx -0.6625$ is also a solution.
So $W(x) = 0$ has $3$ solutions in $(-\frac{\pi}{2},\frac{\pi}{2})$.

In Figures \ref{fig:0-4-h=-4} and \ref{fig:0-4-h=-4P2} we plot the nodal sets of the eigenfunction $\Phi_{0,4,-4,\theta}(x,y)$ for various values of $\theta$.

We now discuss the values of $\theta$ that are presented in Figure \ref{fig:0-4-h=-4}.
The majority of these values corresponds to a change in the number of interior critical zeros or in the number of boundary  zeros, and hence to a change in the number of nodal domains (see \cite{Ley}).

\begin{itemize}
  \item $\theta_1$ is obtained from \eqref{eq:7a} with $y = \frac{\pi}{2}$ and $x \approx 0.6625$.
  \item $\theta_2$ is obtained from \eqref{eq:7a} with $y \approx 0.6625$ and $x = 0$.
  \item $\theta_{11}$ is obtained from \eqref{eq:7a} with $y = \frac{\pi}{2}$ and $x = 0$.
\end{itemize}

 \begin{figure}
\centering
\subfloat[ From left to right: $\theta_0=0$, $\theta_1\approx 0.0264$, $ \theta_2 = (\theta_1 + \theta_3)/2 \approx0.0593$ . \label{0-4-P1}]{%
  \includegraphics[width=\textwidth]{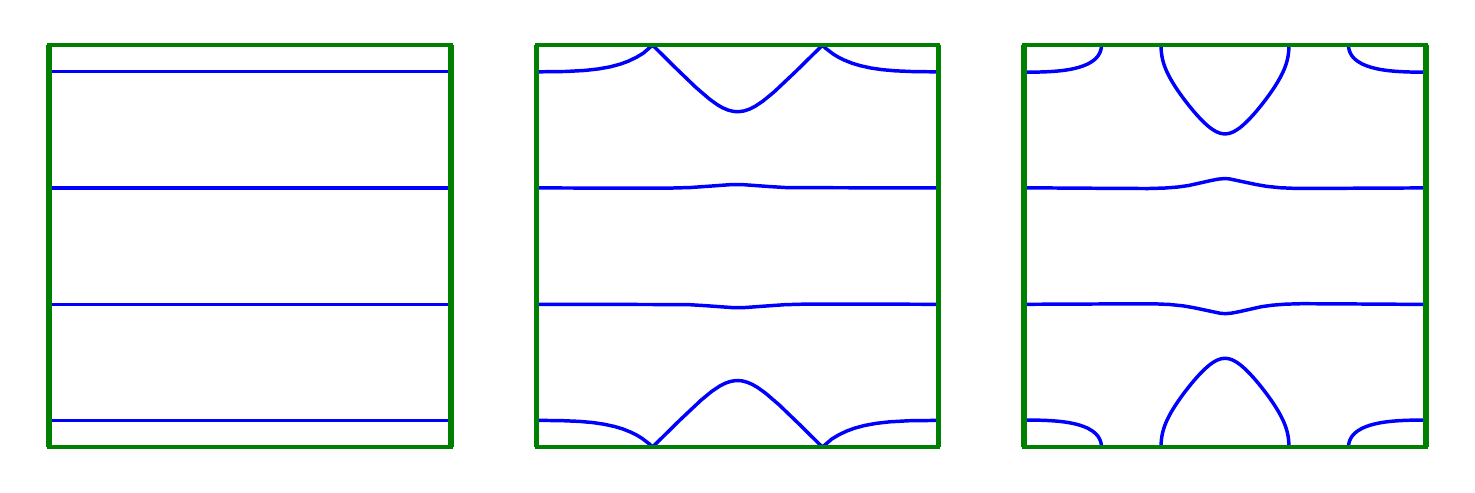}%
  }\par
\subfloat[ From left to right: $\theta_3 \approx 0.0921$, $\theta_4 = \frac{\pi}{4}$, $\theta_5= \frac{\pi}{2}-\theta_{3}\approx 1.4787$ . \label{0-4-P2}]{%
  \includegraphics[width=\textwidth]{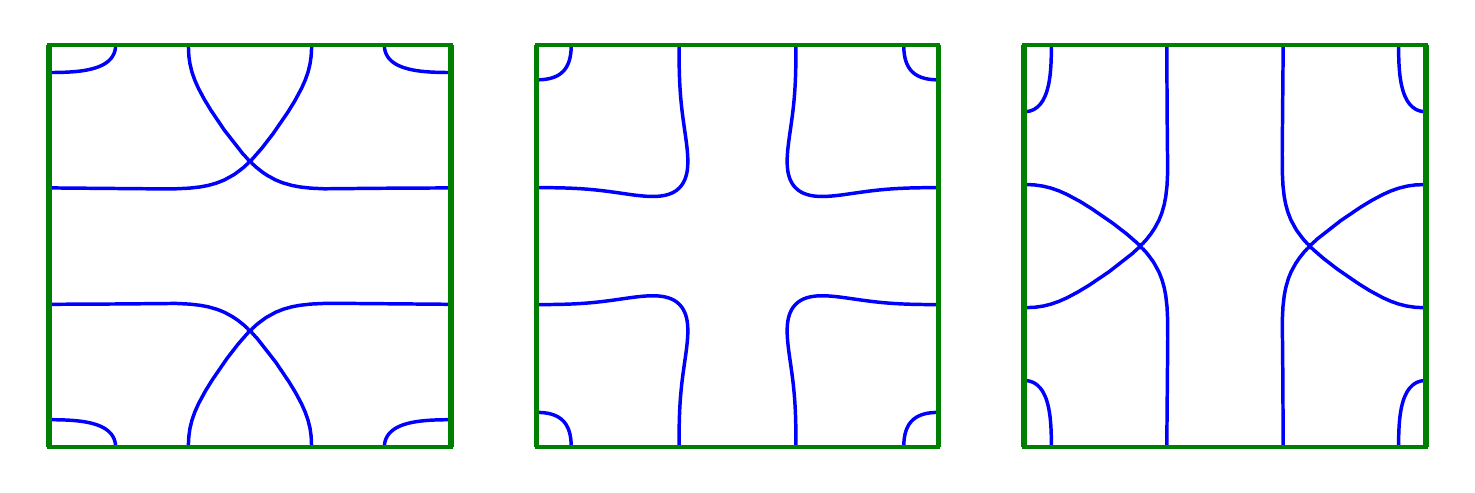}%
  }\par
  \subfloat[ From left to right: $\theta_6 = \frac{\pi}{2} - \theta_2 \approx 1.5115$, $\theta_7 = \frac{\pi}{2}-\theta_{1} \approx 1.5444$, $\theta_8 = \frac{\pi}{2}$ . \label{0-4-P3v3}]{%
  \includegraphics[width=\textwidth]{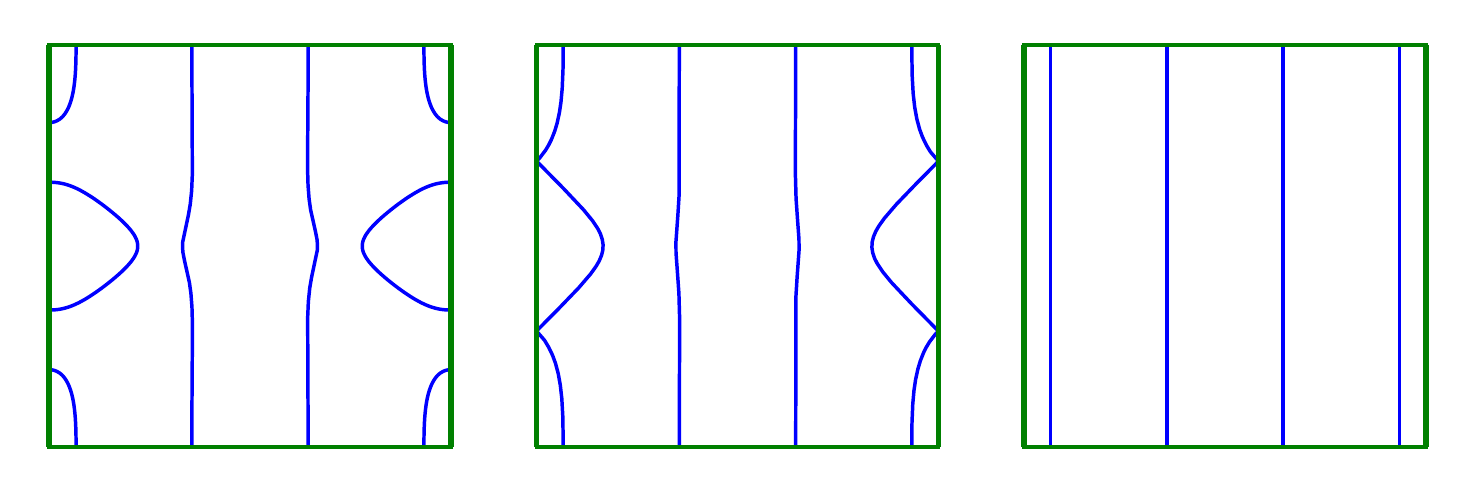}%
  }\par
\subfloat[ From left to right: $\theta_9= \frac{\pi}{2} + (\pi - \theta_{11}) \approx 1.5732$, $\theta_{10}=\frac{5\pi}{8}$, $\theta_{11} = \frac{3\pi}{4}$ . \label{0-4-P4}]{%
  \includegraphics[width=\textwidth]{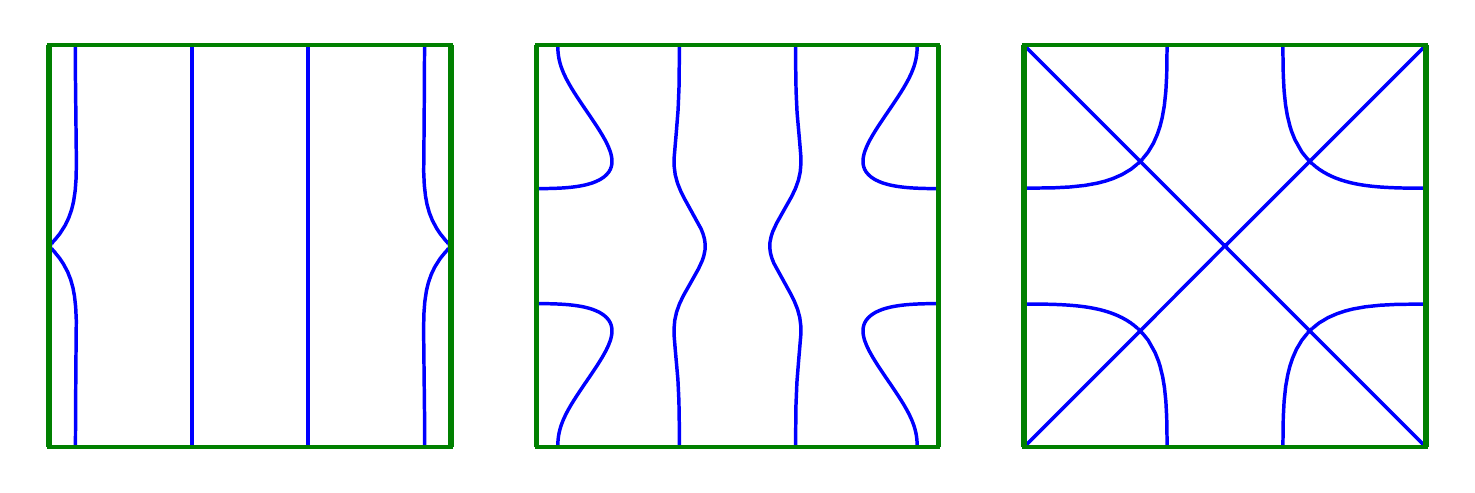}%
  }
    \caption{The nodal sets of the Robin eigenfunction $\Phi_{0,4,h,\theta}$ for $h=-4$ and various values of $\theta$.}
\label{fig:0-4-h=-4}
\end{figure}

 \begin{figure}
\centering
\renewcommand{\thesubfigure}{e}
\subfloat[ From left to right: $\theta_{12}=\frac{7\pi}{8}$, $\theta_{13}\approx 3.1392$. \label{0-4-P5}]{%
  \includegraphics[width=0.65\textwidth]{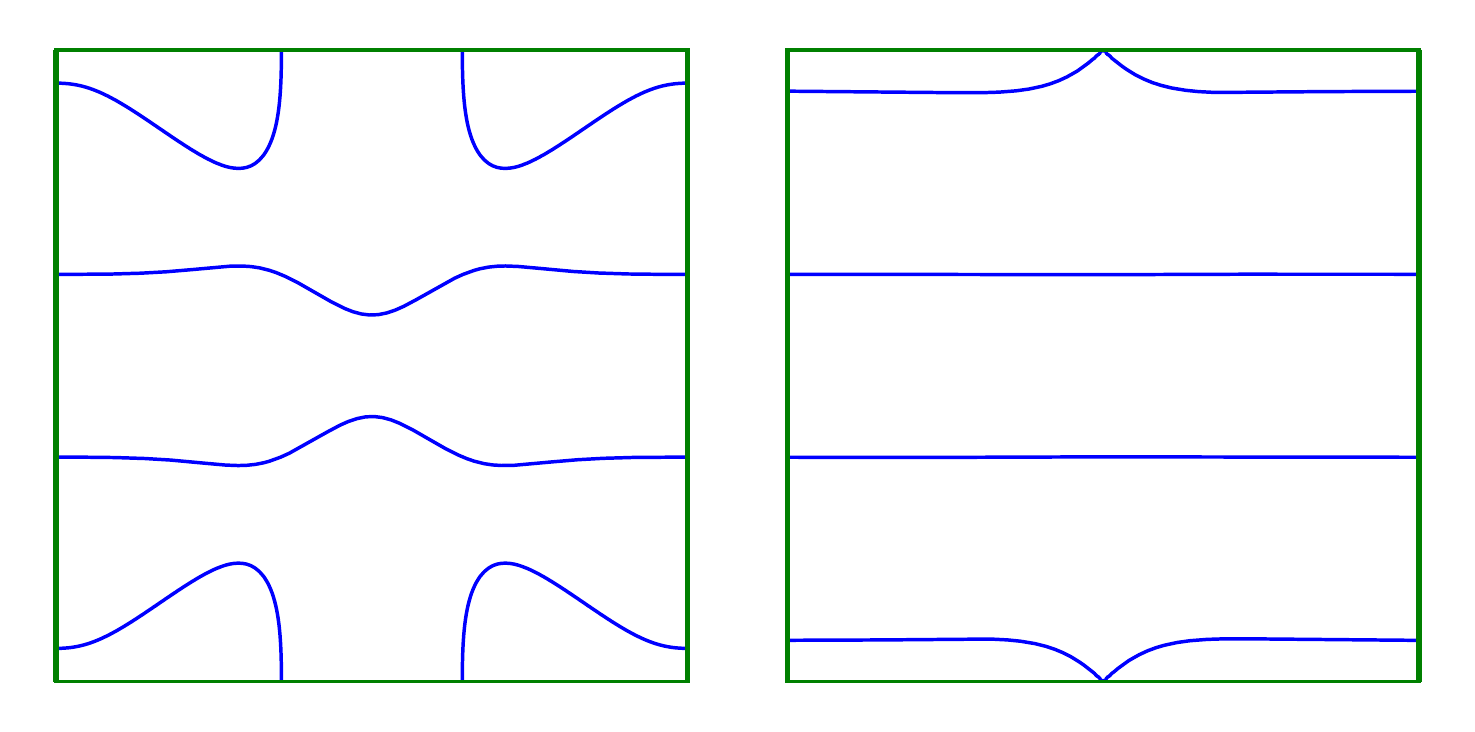}%
  }
\caption{The nodal sets of the Robin eigenfunction $\Phi_{0,4,h,\theta}$ for $h=-4$ and various values of $\theta$.}
\label{fig:0-4-h=-4P2}
\end{figure}
We observe that the number of nodal domains is $5, 9, 11, 9, 11, 9, 5, 7, 12, 7$.

\subsubsection{Complete analysis for $q=4$.}

 By Lemma \ref{lem:ubdzerosW}, $W$ has exactly three solutions $\{\gamma,0,-\gamma\}$ with $\gamma \neq 0$. A priori, there are nine situations to consider. But, using the symmetry property observed in  Remark \ref{rem4sym}, there are only four situations to consider $(\gamma,\gamma)$, $(0,0)$, $(\gamma,0)$, $(0,\gamma)$.

\begin{itemize}
 \item The two first cases lead to the same condition on $\theta$: $\tan \theta=-1$.  So for $\theta=\frac{3\pi}{4}$ there are five  interior critical zeros: $(0,0)$, two lying on the diagonal $y=x$ and two lying on the anti-diagonal $y=-x$.

 \item The third case corresponds to two interior critical zeros lying on the $x$ axis with $\tan \theta = -\cosh (\beta_0 \gamma/\pi)/ \cos ( \alpha_q \gamma/\pi)$.

 \item The fourth case corresponds to two  interior critical zeros lying on the $y$ axis with $ \tan \hat{\theta} = -\cos ( \alpha_q \gamma/\pi)/\cosh (\beta_0 \gamma/\pi)$ and is similar to the third one.
   Indeed, $\theta= \frac{\pi}{2} - \hat{\theta}$ (see Figure \ref{fig:0-4-h=-4}).
\end{itemize}

 Note that $\tan \theta \neq 1$ in all these critical cases.

 Hence, there are only three cases to analyse and this corresponds to what is shown in Figures \ref{fig:0-4-h=-4} and \ref{fig:0-4-h=-4P2} for $h=-4$.

 By Lemma \ref{lemnondeg}, in these three situations the interior critical zeros are non-degenerate and we apply Euler's formula.

 When $\theta =\frac{3\pi}{4} $, as there are at most $4(q-1)$ boundary zeros (counted with multiplicity) by
 Lemma \ref{lem:bdrypts}.
 We obtain
 $$
  k \leq 1 +  5 + \frac 12 (4 \times 3) = 12 < 13\,.
 $$
 Here it is natural to interpret $5$ as $2 (q-2)+1$ and to write it in the form $$k\leq 1 + 2 (q-2) +1  + 2 (q-1)= 4q -4 < 4q -3$$ which leads to a natural conjecture.

 For the two other cases, there are only two interior critical zeros. The corners are not in the zero set so there are at most $4q=16$ boundary zeros (counted with multiplicity) by Lemma \ref{lem:bdrypts}.
 We have that
 $$
 k \leq 1 + 2 + 8 =11 < 13\,.
 $$
 As observed above, there are no interior critical zeros for  $\theta =\frac \pi 4$.

 Hence we have proven the following proposition.
 \begin{prop}
 If $\lambda_{0,4,h}(S)$ is negative, then no associated eigenfunction is Courant-sharp.
 \end{prop}

 We note that in Figure \ref{fig:0-4-h=-4} all the upper bounds are optimal in some cases.
  For example, for $\theta=\frac{3\pi}{4}$ there are $12$ nodal domains.

 \subsection{Conclusion for the pairs $(0,q)$, $q$ even.}
  Our main conjecture  is the following.
 \begin{conjecture}
   For $q$ even, $q \geq 4$ and $\lambda_{0,q,h}(S)$ negative, no associated eigenfunction is Courant-sharp.
 \end{conjecture}
 This has been proved for $q=4$ and for $q>4$ with $h <h_q$.
 We also note that in all the cases that we were able to analyse, the maximal number of nodal domains is achieved in the case where $\theta=\frac{3\pi}{4}$.
 It is natural to ask if this is always true under the assumptions of the conjecture.

\section{Upper bound for positive Courant-sharp Robin eigenvalues of the square.}\label{s:pos}
The analysis of the positive eigenvalues of the Robin Laplacian on $S$ with parameter $h$ is much more delicate because these eigenvalues could have multiplicity larger than $2$ for some values of $h$.
Hence we can only obtain rather weak results in this case.

\begin{prop}\label{p:upbd}
  For $h<0$, the positive Courant-sharp Robin eigenvalues of the  Laplacian on $S$  are bounded.
  More precisely, if $\lambda_{k,h}(S)$ is a Courant-sharp Robin eigenvalue of the Laplacian on $S$, then
  \begin{equation*}
    \lambda_{k,h}(S) < 1091.
  \end{equation*}
\end{prop}

  To prove Proposition \ref{p:upbd}, we adapt the $h$-independent arguments from Section 3 of \cite{GHRS}, which in turn were based on the proof of Proposition 2.1 in \cite{HPS1}.

  \begin{proof}
    For $\lambda > 0$, we denote the number of positive Robin eigenvalues of the Laplacian on $S$ that are strictly smaller than $\lambda$ by $N_{+}^{R,h}(\lambda)$ and we have that
    \begin{align*}
      N_{+}^{R,h}(\lambda) &= \# \{k \in \N : k \geq 1, 0 < \lambda_{k,h}(S) < \lambda\}\\
      &\geq \# \{ (i,j) \in \N^2 : i,j \geq 2, \pi^{-2}(\alpha_i(h)^2 + \alpha_j(h)^2) <\lambda\}.
    \end{align*}
    Since for $h<0$ $\pi^{-2}(\alpha_i(h)^2 + \alpha_j(h)^2) \leq \pi^{-2}(\alpha_i(0)^2 + \alpha_j(0)^2) =i^2 + j^2$,
    we have that
    \begin{align*}
      N_{+}^{R,h}(\lambda)
      &\geq \# \{ (i,j) \in \N^2 : i,j \geq 2, i^2 + j^2 <\lambda\}\\
      & \geq  \frac{\pi}{4}\lambda - 4\sqrt{\lambda}.
    \end{align*}

    If $\lambda_{k,h}(S)$ is Courant-sharp, then $\lambda_{k,h} > \lambda_{k-1,h}$, and
    \begin{equation}\label{eq:Pl2}
      k > N_{+}^{R,h}(\lambda_{k,h}) \geq \frac{\pi}{4}\lambda_{k,h} - 4\sqrt{\lambda_{k,h}}.
    \end{equation}
    By following the arguments from Subsections 3.2 and 3.3 of \cite{GHRS}, we obtain
    \begin{equation}\label{eq:Pl3}
      \frac{\pi}{{\bf j}^2} > \frac{\pi}{4} - \frac{8}{\sqrt{\lambda_{k,h}}}
    \end{equation}
    (instead of inequality (3.13) of \cite{GHRS}). We observe that the mapping
    $$ \lambda \mapsto f(\lambda):= \frac{\pi}{4} - \frac{\pi}{{\bf j}^2} - \frac{8}{\sqrt{\lambda}}$$
    is increasing for $\lambda >0$, and that $f(1090)<0$ while $f(1091)>0$. Hence, if $\lambda_{k,h} \geq 1091$, then \eqref{eq:Pl3} is violated.
  \end{proof}

  We now use Proposition \ref{p:upbd} to show that the number of positive Courant-sharp Robin eigenvalues of the Laplacian on $S$ is bounded independently of $h<0$.

  \begin{prop}\label{p:finpos}
   There exists $N >0$ such that, for any $h<0$, the number of positive Courant-sharp Robin eigenvalues of the Laplacian  on $S$ with parameter $h$ is less than $N$.
  \end{prop}

  \begin{proof}
  From the previous proposition, we can deduce that there are finitely many positive Robin eigenvalues of the form
  $$ \pi^{-2} (\alpha_i(h)^2 + \alpha_j(h)^2), \quad i,j, \geq 2.$$
  We denote the number of such eigenvalues by $N_{+}(h)$ and we search for a uniform  upper bound for this quantity  when $h<0$.
  Since $(i-1)^2 + (j-1)^2 \leq \pi^{-2} (\alpha_i(h)^2 + \alpha_j(h)^2)$, we have that for $\lambda>0$,
  \begin{align*}
  &\# \{ (i,j) \in \N^2 : i,j \geq 2, \pi^{-2}(\alpha_i(h)^2 + \alpha_j(h)^2) <\lambda\}\\
  &\leq \# \{ (i,j) \in \N^2 : i,j \geq 2, (i-1)^2 + (j-1)^2 <\lambda\} \leq \frac{\pi}{4}\lambda.
  \end{align*}
  So we have a uniform bound for $N_+(h)$
  \begin{equation*}
    N_{+}(h) \leq \frac{\pi}{4} \lambda_{k,h} < \frac{\pi}{4} \times 1091 < 857.
  \end{equation*}

  We now suppose that $-\frac{2}{\pi} < h <0$. We are interested in the number of $q \in \N^*$ such that
  $$ 0< \pi^{-2} (-\beta_0(h)^2 + \alpha_q(h)^2) <1091.$$
  We note that
  $$ \lambda_{0,q+1,h}(S) - \lambda_{0,q,h}(S) = \pi^{-2}(\alpha_{q+1}(h)^2 - \alpha_{q}(h)^2),$$
  and by Lemma \ref{lem:monotonicity} (ii), $h \mapsto \alpha_{q+1}(h)^2 - \alpha_{q}(h)^2$ is decreasing.
  So
  $$ \lambda_{0,q+1,h}(S) - \lambda_{0,q,h}(S) \geq \pi^{-2}(\alpha_{q+1}(0)^2 - \alpha_{q}(0)^2)
  =2q+1 \geq 3.$$
  Hence in the interval $(0,1091)$, there are at most $\frac{1091}{3} = 363.67 < 364$ positive Robin eigenvalues corresponding to pairs $(0,q)$ for $-\frac{2}{\pi} < h <0$.

  Next we suppose that $ h \leq -\frac{2}{\pi}$. We are interested in the number of $q \geq 2$ and $q' \geq 2$ such that
  $$ 0< \pi^{-2} (-\beta_0(h)^2 + \alpha_q(h)^2)< 1091, \quad 0 < \pi^{-2} (-\beta_1(h)^2 + \alpha_{q'}(h)^2)<1091.$$
  We note that
  $$ \lambda_{0,q+1,h}(S) - \lambda_{0,q,h}(S) = \pi^{-2}(\alpha_{q+1}(h)^2 - \alpha_{q}(h)^2) = \lambda_{1,q+1,h}(S) - \lambda_{1,q,h}(S),$$
  and by Lemma \ref{lem:monotonicity} (i), $h \mapsto \alpha_{q+1}(h)^2 - \alpha_{q}(h)^2$ is increasing.
  So
  $$ \lambda_{0,q+1,h}(S) - \lambda_{0,q,h}(S) \geq \pi^{-2}(\alpha_{q+1}(-\infty)^2 - \alpha_{q}(-\infty)^2)
  =2q - 1 \geq 3,$$
  and similarly $\lambda_{1,q+1,h}(S) - \lambda_{1,q,h}(S) \geq 3$.
  Hence in the interval $(0,1091)$, there are at most $\frac{2 \times 1091}{3} = 727.33 < 728$ positive Robin eigenvalues corresponding to pairs $(0,q)$ or $(1,q)$ for $h < -\frac{2}{\pi}$.
  \end{proof}

  Finally we show that the number of Courant-sharp Robin eigenvalues of the Laplacian on $S$ is bounded independently of $h<0$.
  \begin{prop}\label{p:fin}
   There exists $\tilde{N} >0$ such that, for any $h<0$, the number of Courant-sharp Robin eigenvalues of the Laplacian on $S$ with parameter $h$ is less than $\tilde{N}$.
  \end{prop}

  \begin{proof}
    We denote the number of negative Robin eigenvalues of the Laplacian on $S$ with parameter $h<0$ by $N_{-}(h)$.
    We observe that the number of negative Robin eigenvalues of $S$ is a decreasing function of $h$.
    So the number of Courant-sharp Robin eigenvalues of $S$ is bounded from above by
    \begin{equation*}
      N + N_{-}(h) \leq N + N_{-}(h^*),
    \end{equation*}
    for $h^* \leq h < 0$, and by $N + 4$ for $h < h^*$ by Theorem \ref{thm:negnotCS} and Theorem \ref{thm:5}.
  \end{proof}

\end{document}